\pdfoutput=1
\documentclass[12pt]{amsart}
\usepackage{amssymb}
\usepackage{fullpage}
\usepackage{graphicx}
\usepackage{pictexwd,dcpic}
\usepackage{wrapfig}
\usepackage{enumitem}
\usepackage{url}

\theoremstyle{plain}
\newtheorem{thm}{Theorem}[]
\newtheorem{prop}[thm]{Proposition}
\newtheorem{cor}[thm]{Corollary}
\newtheorem{lem}[thm]{Lemma}
\newtheorem{conj}[thm]{Conjecture}
\newcommand{\comm}[1]{}
\def\l{{\lambda}}
\theoremstyle{definition}
\newtheorem{defi}[thm]{Definition}
\theoremstyle{remark}
\newtheorem{rem}[thm]{Remark}
\newtheorem{exas}[thm]{Examples}
\newtheorem{exa}[thm]{Example}
\begin{document}
\title{Categorified $\mathfrak{sl}_N$ invariants of colored rational tangles}
\author{Paul Wedrich}
\address{Centre for Mathematical Sciences, University of Cambridge, CB3 0WB, England}
\email{P.Wedrich@dpmms.cam.ac.uk}
\begin{abstract}
We use categorical skew Howe duality to find recursion rules that compute categorified $\mathfrak{sl}_N$ invariants of rational tangles colored by exterior powers of the standard representation. Further, we offer a geometric interpretation of these rules which suggests a connection to Floer theory. Along the way we make progress towards two conjectures about the colored HOMFLY homology of rational links.
\end{abstract}
\maketitle
\setcounter{page}{1}

%% update bibliography - check whether things have been accepted

%%done change notation for sl
%%done check definition with sabin
%%done add Hopf link computation -- argue why should extend to invariant of knotted webs
%%done link to survey in intro
%%done Account for Sabin's new paper on rigidity
%%done see whether paper still compatible with update to Clasp technology paper
%%done formatting including eqnarrays
%%done link to my website including mathematica nb
%%done Ciprian's comments
%%done-minimalistically Catharina's comments 
%%done Change exposition about spider category along CKM paper
%%done Consistency in \cong and = at twist rules 

\section{Introduction and statement of results}

Reshetikhin and Turaev \cite{ReT} define invariants of framed oriented tangles (and thus also knots and links) with components labelled (ˋcolored´) by irreducible representations of a semi-simple Lie algebra. Starting from the work of Khovanov \cite{Kh1} and Khovanov-Rozansky \cite{KR1} on case of $\mathfrak{sl}_N$ with the standard representation, much of the $\mathfrak{sl}_N$ package of Reshetikhin-Turaev invariants has been categorified using a variety of different methods, for a recent survey see e.g. \cite{Tur}. The best studied case is the one of fundamental (minuscule) $\mathfrak{sl}_N$ representations, i.e. the exterior powers $\Lambda^k$ of the standard representation. 

On the decategorified level it is well known that $\mathfrak{sl}_N$ Reshetikhin-Turaev invariants stabilize for $N\to \infty$ and hence can be interpreted as specializations of two-variable HOMFLY-type invariants via setting $a=q^N$. Khovanov-Rozansky's HOMFLY homology \cite{KR2}, \cite{Kh2} is a categorification of the HOMFLY polynomial and an extension to \emph{colored HOMFLY homology} with respect to labellings $\Lambda^k$ was developed by Mackaay, Sto\v{s}i\'c and Vaz \cite{MSV} and proved to be a link invariant by Webster and Williamson \cite{WW}.

Colored HOMFLY homology is poorly understood even for the simplest non-trivial knots and links. A good starting point to understand link homologies are rational knots and links, which have been proven to have particular simple uncolored $\mathfrak{sl}_N$ and HOMFLY homologies --- they are essentially determined by decategorified invariants. 
This paper is guided by two conjectures about the colored HOMFLY homology of rational knots and links.

\begin{conj} 
\label{conjA}
`Changing color on unknot components only shifts q-grading.'\\
Let $L$ be a rational two-component link with components colored by a fixed color $\Lambda^j$ and a variable color $\Lambda^i$ and let $HP_{j}(L,\Lambda^i)\in \mathbb{N}[a^{\pm 1}, t^{\pm 1}, q^{-1}][[q]]$ be the Hilbert-Poincar\'e series of the $\Lambda^i$-reduced $(\Lambda^i,\Lambda^j)$-colored HOMFLY homology of $L$.
 
Then there exist rational functions $F_{j}(L)\in \mathbb{Z}[a^{\pm 1}, s^{\pm 1}, t^{\pm 1}](q)$ such that for all $i\geq j$
\[HP_{j}(L,\Lambda^i) = F_{j}(L)(a, s= q^{i-j},t , q)\]
after expanding the right hand side into a power series in $q$.
\end{conj}

Conjecture 1 is due to Marko Sto\v{s}i\'c (private communication), see also \cite{GNSS}. 

\begin{conj}
\label{conjB}
`Homologies of higher colors are like powers of uncolored homology.'\\
Let $L$ be a rational knot or link. There exists a corrected $Q$-grading on reduced triply-graded $\Lambda^j$-colored HOMFLY homology with Poincar\'e polynomials $\tilde{P}_j(L)(a,Q,t)$, such that 
\[\tilde{P}_j(L)(a,Q,t)=(\tilde{P}_1(L)(a,Q,t))^j.\]
In particular, the ordinary Poincar\'e polynomials $P_j(L)(a,q,t)$ of $\Lambda^j$-colored HOMFLY homology satisfy
 \[P_j(L)(a,1,t)=(P_1(L)(a,1,t))^j.\] 
\end{conj}
This is a special case of the `refined exponential growth' conjecture in \cite{GGS}, where a similar behaviour is conjectured for rational knots and torus knots in the more general setting of labellings by rectangular Young diagrams. Moreover, there the correction $Q$ of the standard $q$-grading results from comparing two alternative homological gradings.\\

While we cannot prove these conjectures for rational knots and links, we make progress towards them by establishing related results for categorified $\mathfrak{sl}_N$ invariants of colored rational tangles. Rational tangles $T(p,q)$ are built recursively, starting from a trivial two-strand tangle, by planar composition with crossings on the top and on the right, according to the continued fraction expansion of $\frac{p}{q}$. All rational links are closures of rational tangles. We give an example:
\[\text{The tangle } T(7,2)\text{ associated to }\frac{7}{2}=[3,2]\text{ is } \quad \vcenter{\hbox{\includegraphics[height=1.0cm, angle=0]{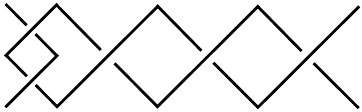}}} \]
Using the categorical skew Howe duality framework developed by Cautis, Kamnitzer and Licata \cite{Cau2}, \cite{CKL2} we study categorified $\mathfrak{sl}_N$ invariants of rational tangles with strands labelled by exterior powers of the standard representation. These invariants take values in a homotopy 2-category of chain complexes over categories that depend on the colors on the strands of the tangle, see section \ref{categ}. In section \ref{rattan} we demonstrate that for fixed colors $\Lambda^i$ and $\Lambda^j$ these invariants can be written as complexes with chain spaces being direct sums of grading shifts of a finite number of basic objects represented by webs (MOY graphs). We use results from \cite{Cau2} to compute explicit twist rules that describe how the invariant of a rational tangle changes under composing with a crossing on the top or on the right. The following example illustrates twist rules in the special case of Bar-Natan's geometric version of Khovanov homology for tangles \cite{BN}.

\begin{exa} The Khovanov homology of a positive rational tangle can be computed recursively by using the following twist rules which describe what happens to a generator of a chain space under tensoring with a crossing complex. We write top- and right twists as operators $T$ and $R$. 
\begin{align*}
  T(\vcenter{\hbox{\includegraphics[height=0.4cm, angle=0]{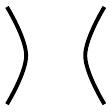}}})\quad&=\quad  \vcenter{\hbox{\includegraphics[height=0.4cm, angle=0]{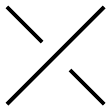}}} \quad=\quad 0\to  q t^{-1}~ \vcenter{\hbox{\includegraphics[height=0.4cm, angle=90]{KUP.pdf}}} \to  \vcenter{\hbox{\includegraphics[height=0.4cm, angle=0]{KUP.pdf}}}\to 0\\
  T(\vcenter{\hbox{\includegraphics[height=0.4cm, angle=90]{KUP.pdf}}})\quad&=\quad \vcenter{\hbox{\includegraphics[height=0.4cm, angle=90]{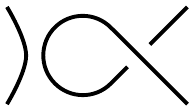}}}\quad =\quad 0\to  q t^{-1}~ \vcenter{\hbox{\includegraphics[height=0.4cm, angle=90]{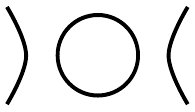}}} \to  \vcenter{\hbox{\includegraphics[height=0.4cm, angle=90]{KUP.pdf}}}\to 0 ~\quad\sim \quad 0 \to  q^2 t^{-1}~ \vcenter{\hbox{\includegraphics[height=0.4cm, angle=90]{KUP.pdf}}}\to 0\to 0\\ 
   R(\vcenter{\hbox{\includegraphics[height=0.4cm, angle=0]{KUP.pdf}}})\quad&=\quad \vcenter{\hbox{\includegraphics[height=0.4cm, angle=0]{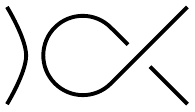}}}~=\quad 0\to  \vcenter{\hbox{\includegraphics[height=0.4cm, angle=0]{KUP.pdf}}} \to  q^{-1} t~ \vcenter{\hbox{\includegraphics[height=0.4cm, angle=0]{Kcirc.pdf}}}\to 0 \quad\sim\quad 0 \to 0\to q^{-2} t~ \vcenter{\hbox{\includegraphics[height=0.4cm, angle=0]{KUP.pdf}}}\to  0\\
R(\vcenter{\hbox{\includegraphics[height=0.4cm, angle=90]{KUP.pdf}}})\quad&=\quad \vcenter{\hbox{\includegraphics[height=0.4cm, angle=0]{Kcr.pdf}}}\quad=\quad 0\to  \vcenter{\hbox{\includegraphics[height=0.4cm, angle=90]{KUP.pdf}}} \to q^{-1} t~ \vcenter{\hbox{\includegraphics[height=0.4cm, angle=0]{KUP.pdf}}}\to 0
\end{align*}
Here all possible non-trivial differentials are given by saddle cobordisms and compared to \cite{BN} we use slightly different grading conventions that are more compatible with the colored $\mathfrak{sl}_N$ case treated in this paper. 
\end{exa}

To be more precise, a priori the rules in the above example only compute the chain spaces and the component of the differential coming from the last added crossing. To get the full information about the complex that is the tangle invariant, one would need to compute the induced differentials coming from previous crossings.

\begin{defi} The \emph{colored HOMFLY complex} of a colored rational tangle $T$ is the chain complex representing the categorified Reshetikhin-Tuarev $\mathfrak{sl}_N$ invariant of $T$ that is obtained recursively by applying the twist rules from section \ref{twistrules} term-wise and computing the induced differentials.
\end{defi}

In the general case of a $\mathfrak{sl}_N$ colored HOMFLY complex we find it difficult to compute induced differentials and thus restrict our attention to counting the number and grading shifts of basic objects in its chain spaces. The correct data structure is, therefore, the following modified notion of Poincar\'e polynomial.

\begin{defi}
The \emph{Poincar\'e polynomial} $\mathcal{P}^N_{i,j}(T)$ of the $(\Lambda^i,\Lambda^j)$-colored HOMFLY complex of a rational tangle $T$ with $i\geq j$ is an element of $\mathbb{N}[q^{\pm 1},t^{\pm 1}]\langle X_0,\dots, X_j\rangle $, the free $\mathbb{N}[q^{\pm 1},t^{\pm 1}]$ module (over a semiring) spanned by basic objects $ X_k$. It is the formal sum of the basic objects $X_k$ appearing in the chain spaces of the complex, weighted by powers of $t$ and $q$ indicating shifts in homological and $q$-grading respectively. The basic objects $X_k$ of weight $k$ are introduced in Definition \ref{BasicMOY}.
\end{defi}

\begin{exa}
Bar-Natan's complex for the $(2,k)$ torus tangle $T(k,1)$ is (with the grading convention introduced above):
\[ q^{2k-1}t^{-k}~ \vcenter{\hbox{\includegraphics[height=0.4cm, angle=90]{KUP.pdf}}} \to \cdots \to q^{3}t^{-2}~ \vcenter{\hbox{\includegraphics[height=0.4cm, angle=90]{KUP.pdf}}} \to q t^{-1}~ \vcenter{\hbox{\includegraphics[height=0.4cm, angle=90]{KUP.pdf}}} \to  \vcenter{\hbox{\includegraphics[height=0.4cm, angle=0]{KUP.pdf}}}\]
The Poincar\'e polynomial of this complex is $ \vcenter{\hbox{\includegraphics[height=0.4cm, angle=0]{KUP.pdf}}} + \sum_{i=1}^k q^{2i-1}t^{-i}~ \vcenter{\hbox{\includegraphics[height=0.4cm, angle=90]{KUP.pdf}}}$.
\end{exa}

The twist rules for the case of colored $\mathfrak{sl}_N$ invariants of rational tangles, which we compute in section \ref{twistrules}, stabilize in an obvious way for large $N$ and exhibit a very simple dependence on the higher color $i$, thus immediately implying the following theorem.

\begin{thm}\label{thmA} The $\mathfrak{sl}_N$ colored HOMFLY complexes of a positive rational tangle $T$ colored by representations $\Lambda^i$ and $\Lambda^j$ with $N\gg i\geq j$, depend essentially only on $j$ up to shifts in $q$-grading.
More precisely, there exist elements $\mathcal{P}_j(T)\in \mathbb{N}[a^{\pm 1},q^{\pm 1},s^{\pm 1},t^{\pm 1}]\langle X_0, \dots X_j \rangle $, such that:
\[\mathcal{P}^N_{i,j}(T) = \mathcal{P}_j(T)(a=q^N, q, s=q^{i-j}, t).\]
\end{thm}
We, thus, sometimes consider $a$ and $s$ as two additional gradings of basic objects in the colored HOMFLY complex.

In section \ref{conj2} we show that in the decategorified setting a statement similar to Theorem \ref{thmA} holds for any link with an unknot component:
\begin{prop} \label{colshiftprop}
Let $L'=L \cup U$ be a link with an unknot component $U$ and some fixed coloring on the components of $L$. Then there exists $P^{st}(L',U)\in \mathbb{Z}[a^{\pm 1},s^{\pm 1}](q)$, such that for any $i\in \mathbb{N}$: $P^{st}(L',U)(a,s=q^i,q)$ is the colored HOMFLY polynomial of $L$ with color $\Lambda^i$ on $U$, reduced with respect to $\Lambda^i$.\\
We call $P^{st}(L',U)$ the \emph{color stable HOMFLY polynomial} of $L'$ with respect to the unknot component $U$
\end{prop}

In section \ref{geom} we propose a geometric algorithm for computing the colored HOMFLY complex of a colored rational tangle, which is very similar to Bigelow's geometric model for the Jones polynomial \cite{Big1}. The basis for this geometric algorithm is a special planar diagram $P$ of the tangle in a disc $D^2\subset \mathbb{C}$, displaying only one strand $\alpha$. In the case of colors $\Lambda^i$ and $\Lambda^j$ with  $i\geq j$ we establish a bijection between the generators of the colored HOMFLY complex and intersections of $Sym^j(\alpha)$ with certain half-dimensional submanifolds $V_h$ in $Sym^j(D^2)$. We further expect that in this picture differentials are realized as discs connecting intersection points with boundary on $V_h$ and $Sym^j(\alpha)$. Their gradings are determined by graded intersection numbers with certain divisors in $Sym^j(D^2)$. In particular, using the language from Conjecture \ref{conjB}, we identify the difference between the original $q$-grading and the corrected $Q$-grading as coming from intersection of such discs with the big diagonal in $Sym^j(D^2)$.
 \setlength\intextsep{0pt}
\begin{wrapfigure}[6]{r}{4.1cm}
\begin{center}
   \includegraphics[height=3cm]{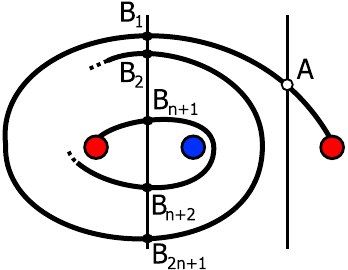}
   \end{center}
\end{wrapfigure} 

The diagram $P$ for the tangle $T(2n+1,1)$ is shown on the right. In the case of colors $(\Lambda^1, \Lambda^1)$ the generators of the colored HOMFLY complex are in bijection with the set $\{A, B_1,\dots, B_{2n+1}\}$. In the case of Bar-Natan's version of Khovanov homology for tangles, $A$ and $B_j$ are generators of type $\hbox{\includegraphics[height=0.4cm, angle=0]{KUP.pdf}}$ and $\hbox{\includegraphics[height=0.4cm, angle=90]{KUP.pdf}}$ respectively.\\
    
In Theorem \ref{mainthm} we prove that this geometric model observes the twist rules from section \ref{twistrules} and, hence, correctly computes the Poincar\'e polynomial of the colored rational tangle. As a corollary we get the following theorem. 

\begin{thm}
\label{thmB} The generators of the $(\Lambda^i,\Lambda^j)$-colored HOMFLY complex of a rational tangle with $i\geq j$ are in bijection with $j$-tuples of generators of the $(\Lambda^i,\Lambda^1)$-colored HOMFLY complex. The $a$-, $s$- and $t$-gradings of these $j$-tuples are additive, i.e. they are the sums of the $a$-, $s$- and $t$-gradings of the components. The $q$-grading, on the other hand, splits into an additive component $Q$ and a non-additive component $q-Q$. With respect to these gradings, colored HOMFLY complexes satisfy an analogue of Conjecture \ref{conjB}. 
\end{thm}

There are obvious structural similarities between our geometric set-up and various Floer theoretic constructions. In particular we want to mention Manolescu's interpretation \cite{Man1} of Seidel-Smith symplectic Khovanov homology \cite{SeS} in terms of Bigelow's model for the Jones polynomial \cite{Big1}, see also section \ref{bigelow}. It is an interesting question whether it is possible to phrase our geometric algorithm in completely symplectic language and hence give an A-model realization of $\mathfrak{sl}_N$ link homologies for colored rational tangles. In the meantime we present a connection to link Floer homology that follows from the geometric model for the colored HOMFLY complex:
\begin{cor}
\label{LFH} Let $T$ be a rational tangle whose denominator closure is a two-component link $L$. Then the $(\Lambda^i,\Lambda^1)$-colored HOMFLY complex of $T$ computes the link Floer homology of $L$.
\end{cor}

In the final section we compare the invariants described in this paper with (multivariable) link invariants arising from Lie superalgebras $\mathfrak{sl}_{m|n}$, as defined by Geer and Patureau-Mirand \cite{GP2}, which satisfy very interesting color stability properties \cite{GP1}.

\subsection*{Structure of this paper}

Section \ref{review} is a review of \cite{Cau2}, \cite{CKM} and related papers. In section \ref{rattan} we compute the twist rules that recursively determine the Poincar\'e polynomials of colored rational tangles and, thus, prove Theorem \ref{thmA}. Section \ref{geom} introduces a picture-way of computing the colored HOMFLY complexes of rational tangles, which proves Theorem \ref{thmB} and suggests a Floer theoretic interpretation. Section \ref{conj2} contains proofs of Proposition \ref{colshiftprop} and Corollary \ref{LFH} and makes contact with multivariable link invariants from Lie superalgebras. 

\subsection*{Acknowledgements}
Some of the new results in this paper are based on ideas of my PhD supervisor Jacob Rasmussen and I have profited greatly from discussions with him. Preliminary work was done by his undergraduate summer student Mihajlo Ceki\'c on the twist rules in the case of colors $(\Lambda^2,\Lambda^2)$.  
I would further like to thank Marko Sto\v{s}i\'c for pointing me towards Conjectures 1 and 2, and Sabin Cautis, Ciprian Manolescu, Bertrand Patureau-Mirand, Catharina Stroppel and Daniel Tubbenhauer for helpful email exchanges and discussions. \footnote{
The author's PhD studies at the Department of Pure Mathematics and Mathematical Statistics, University of Cambridge, are supported by the ERC grant ERC-2007-StG-205349 held by Ivan Smith and an EPSRC department doctoral training grant.}

\section{Categorified quantum link invariants via categorical skew Howe duality}
\label{review}
In this section we set up the framework in which our calculations take place.

\subsection{Reshetikhin-Turaev tangle invariants}
\label{RTi}
\begin{defi} A \emph{tangle} is an embedding of pairs $(T,\partial T)\to (D^2\times I,D^2\times \partial I)$, where $T$ is the disjoint union of a finite number of oriented arcs and circles, $I$ is the interval $[-1,1]$ and $D^2$ is the unit disc in $\mathbb{C}$. The ends of the arcs lying on $D^2\times \{-1\}$ and $D^2\times \{1\}$ are called \emph{bottom and top ends} respectively. A \emph{tangle diagram} is a generic projection of the tangle onto $I\times I$, sending  $D^2\times \{\pm 1\}$ to  $I \times \{\pm 1\}$.\\
We consider tangles up to regular isotopies that fix the boundary. This means we identify tangles that are given by tangle diagrams which differ only by planar isotopy and Reidemeister moves of type 2 and 3. To emphasize that we do not allow Reidemeister moves of type 1, we sometimes refer to tangles as framed tangles. In the following, the arcs and circles of tangles (and hence tangle diagrams) are allowed to carry orientations and labels (colors).  
\end{defi}
 
In \cite{Wit} Witten gave a physical interpretation of the Jones polynomial and the $a=q^N$ specializations of the HOMFLY polynomial via Chern-Simons theory with gauge group $G=SU(2)$ and $G=SU(N)$ respectively. Later Reshetikhin and Turaev \cite{ReT} provided a rigorous mathematical framework for such quantum link invariants using the representation theory of quantum groups $U_q(\mathfrak{g})$, which are deformations of the enveloping algebras of the Lie algebras of the gauge groups $G$. More generally, they define invariants of framed tangles $T$ with strands labelled by irreducible representations of $U_q(\mathfrak{g})$ or equivalently by dominant weights of the semi-simple Lie algebra $\mathfrak{g}$. If the bottom end strands are labelled by $\underline{\lambda}=(\lambda_1,\dots, \lambda_n)$ and the top end strands by $\underline{\mu} = (\mu_1,\dots, \mu_{n'})$, they associate to $T$ a map of $U_q(\mathfrak{g})$-representations
\[\psi(T)\colon V_{\underline{\lambda}}=V_{\lambda_1}\otimes\dots\otimes V_{\lambda_n}\to V_{\mu_1}\otimes\dots\otimes V_{\mu_{n'}}= V_{\underline{\mu}}\]
where $V_\lambda$ denotes the irreducible $U_q(\mathfrak{g})$ representation of highest weight $\lambda$. This map is an invariant of the framed labelled tangle. The invariant is computed by isotoping $T$ into generic position, then scanning it from bottom to top and assembling the map from the maps associated to cups, caps and crossings.\\

In the case of links $L$ and complex semi-simple Lie algebras $\mathfrak{g}$, where the bottom and top representations are both the trivial $U_q(\mathfrak{g})$ representation  $\mathbb{C}(q)$, the map $\psi(L)$ is given by multiplying by $\psi(L)(1)\in \mathbb{C}(q)$. We list some prominent instances:
\begin{itemize}
\item If $\mathfrak{g}=\mathfrak{sl}_{2}$ and $L$ is labelled by (symmetric powers of) the standard representation, then $\psi(L)(1)$ is the (colored) Jones polynomial of $L$.
\item If $\mathfrak{g}=\mathfrak{sl}_N$ and $L$ is labelled by the standard representation, then $\psi(L)(1)$ is the $a=q^N$ specialization of the HOMFLY polynomial of $L$.
\item If $\mathfrak{g}=\mathfrak{sl}_N$ and $L$ is labelled by some other irreducible representations, then $\psi(L)(1)$ is the $a=q^N$ specialization of the corresponding colored HOMFLY polynomial of $L$.
\end{itemize} 
\subsection{The $U_q(\mathfrak{sl}_N)$ representation category and web relations}
In this paper we are only interested in the tangle invariants for $\mathfrak{g}=\mathfrak{sl}_N=\mathfrak{sl}_N(\mathbb{C})$ and fundamental representations $\Lambda^iV$, where $V$ is the standard representation. The representation category of the quantum group $U_q(\mathfrak{sl}_N)$ is well understood and has a nice graphical description which ties in nicely with the Reshetikhin-Turaev picture. We follow the exposition in \cite{CKM}.

\begin{defi}
The \emph{free spider category} $\mathrm{FSp}(\mathfrak{sl}_N)$ has as objects finite sequences $\underline{k}$ in $\{1^\pm,\dots, (N-1)^\pm\}$ and morphisms in $Mor(\underline{k},\underline{l})$  are $\mathbb{C}(q)$-linear combinations of oriented planar graphs, called \emph{webs}, in the unit square $[-1,1]\times [-1,1]\subset \mathbb{R}^2$, up to planar isotopy, such that:
\begin{itemize}
\item The edges of the web are labelled by elements of the set $\{1,\dots ,N-1\}$.
\item The boundary of the web splits into a bottom boundary on $[-1,1]\times\{-1\}$ and a top boundary on $[-1,1]\times\{1\}$.
\item The sequence of labels and orientations on the bottom and top boundary agree with the sequences $\underline{k}$ and $\underline{l}$, where $j^+$ and $j^-$ stand for upward and downward oriented strands labelled by $j$.
\item The internal vertices are of only four types:
\[\vcenter{\hbox{\includegraphics[height=1.2cm, angle=0]{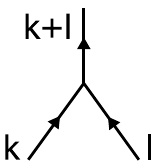}}}\quad \vcenter{\hbox{\includegraphics[height=1.2cm, angle=0]{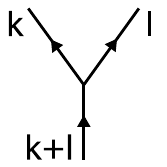}}}\quad \vcenter{\hbox{\includegraphics[height=1.2cm, angle=0]{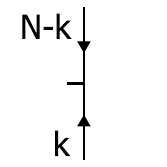}}} \quad \vcenter{\hbox{\includegraphics[height=1.2cm, angle=0]{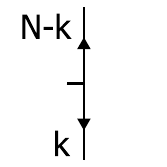}}}\]
\end{itemize} 
The latter two bivalent vertices are called \emph{tags}. In general it does matter on which side the tag is placed and for now we distinguish these local sub-graphs from their mirror images. In some formulas we further allow labellings of edges by integers outside the range of $\{1,\dots,N-1\}$. Edges labelled by $0$ are to be deleted, $N$-labelled edges incident to a trivalent vertex get reduced to tags as shown in the following figure. Any diagram containing a labelling outside the range of $\{0,\dots, N\}$ is defined to be $0$.
\[ \vcenter{\hbox{\includegraphics[height=1.2cm, angle=0]{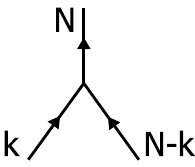}}}=\vcenter{\hbox{\includegraphics[height=1.2cm, angle=0]{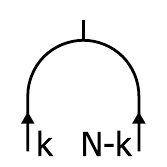}}}\quad \vcenter{\hbox{\includegraphics[height=1.2cm, angle=0]{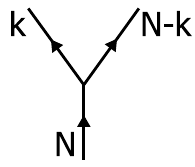}}} = \vcenter{\hbox{\includegraphics[height=1.2cm, angle=0]{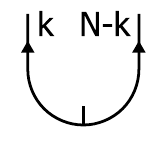}}}\]

Morphisms are composed by vertically stacking graphs and glueing boundary points.\\
 
We further define for $n \in \mathbb{Z}$, $k\in \mathbb{N}$, $k\leq |n|$:
\begin{itemize}
\item  the \emph{quantum integers} $[n]:= \frac{q^{n}-q^{-n}}{q^{1}-q^{-1}}$.
\item  the \emph{quantum factorials} $[n]! := [1][2]\cdots [n]$ if $n\geq 0$ and $[n]! := [-1][-2]\cdots [n]$ if $n\leq 0$.
\item the \emph{quantum binomial coefficients} ${n\brack k }:=\frac{[n][n-1]\cdots [n-k+1]}{[k]!}$.
\end{itemize}
Note that $[-n]=-[n]$ and $[-n]!=(-1)^n[n]$. Negative quantum integers are only used in the last relation in the following definition. For the remainder of this paper we only deal with positive quantum integers.
\end{defi}
\begin{defi}
The \emph{$\mathfrak{sl}_N$ spider category} $\mathrm{Sp}(\mathfrak{sl}_N)$ is the quotient of the free spider category $\mathrm{FSp}(\mathfrak{sl}_N)$ by the following local relations on morphisms

\begin{eqnarray}
\label{rel1} \quad\vcenter{\hbox{\includegraphics[height=1.2cm, angle=0]{Rel12.pdf}}}\quad &=&(-1)^{k(N-k)}\vcenter{\hbox{\includegraphics[height=1.2cm, angle=0]{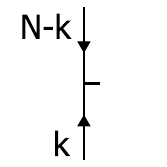}}}\\
\label{rel2}\quad \vcenter{\hbox{\includegraphics[height=2cm, angle=0]{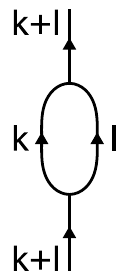}}}\quad&=&{k+l\brack l}\vcenter{\hbox{\includegraphics[height=2cm, angle=0]{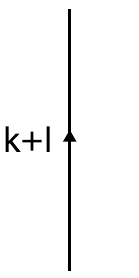}}}\\
\label{rel3} \vcenter{\hbox{\includegraphics[height=2cm, angle=0]{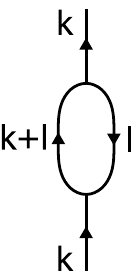}}}\quad&=&{N-k\brack l}\vcenter{\hbox{\includegraphics[height=2cm, angle=0]{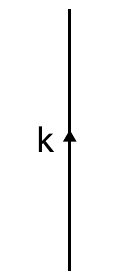}}}\\
\label{rel4}\quad \vcenter{\hbox{\includegraphics[height=2cm, angle=0]{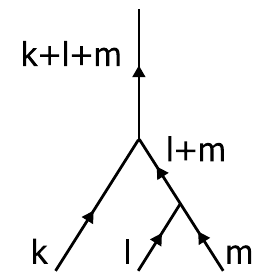}}}&=&\vcenter{\hbox{\includegraphics[height=2cm, angle=0]{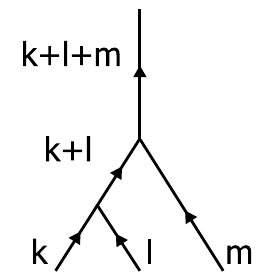}}}\\
\label{rel4b}\quad \vcenter{\hbox{\includegraphics[height=2cm, angle=0]{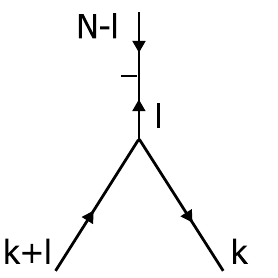}}}&=&\vcenter{\hbox{\includegraphics[height=2cm, angle=0]{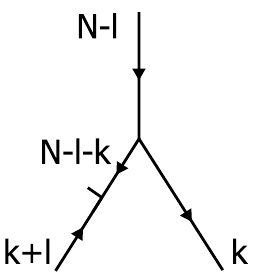}}}\\
\label{rel5}\quad \vcenter{\hbox{\includegraphics[height=2cm, angle=0]{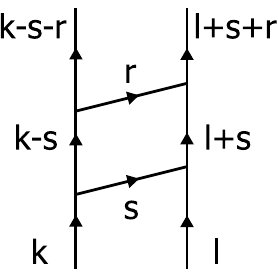}}}&=&{r+s \brack r}\vcenter{\hbox{\includegraphics[height=2cm, angle=0]{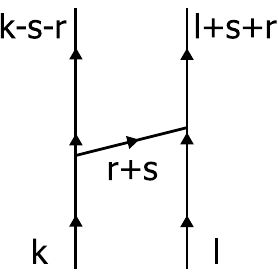}}}\\
\label{rel6}\quad \vcenter{\hbox{\includegraphics[height=2cm, angle=0]{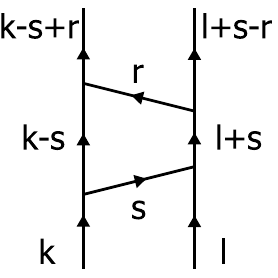}}}&=&\sum_{t} {k-l+r-s \brack t}\vcenter{\hbox{\includegraphics[height=2cm, angle=0]{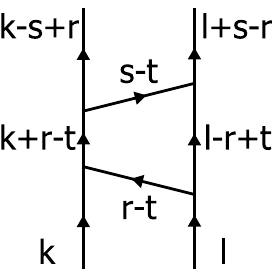}}}
\end{eqnarray}

together with the mirror images and arrow reversals of these. 
\end{defi}

We briefly recall some facts about the representation category of $U_q(\mathfrak{sl}_N)$ and collect them in the following lemma. Details can be found in \cite{CKM}.

\begin{lem}
The representation category of $U_q(\mathfrak{sl}_N)$ is a pivotal category and it is the idempotent completion (Karoubi envelope) of the full subcategory generated by objects isomorphic to tensor products of fundamental representations, which we denote by $\mathrm{Rep}(U_q(\mathfrak{sl}_N))$. \\

As a pivotal category $\mathrm{Rep}(U_q(\mathfrak{sl}_N))$ is generated by up to scaling unique intertwiners
\[M_{k,l} \colon \Lambda^kV\otimes \Lambda^lV \to \Lambda^{k+l}V\]
\[M_{k,l}' \colon \Lambda^{k+l}V\to \Lambda^kV\otimes \Lambda^lV \]

\end{lem}

\begin{thm} \label{ckmthm}(Theorems 3.2.1 and 3.3.1 in \cite{CKM}) There is an equivalence of pivotal categories $\phi\colon  \mathrm{Sp}(\mathfrak{sl}_N) \to \mathrm{Rep}(U_q(\mathfrak{sl}_N))$, defined on
\begin{itemize}
\item objects by: 
\[\underline{k}=(k_1^{\epsilon_1},\dots, k_s^{\epsilon_s})\to (\Lambda^{k_1}V)^{\epsilon_1}\otimes\cdots\otimes(\Lambda^{k_s}V)^{\epsilon_s}\] where $\epsilon_k\in \{-1,1\}$ and $W^{-1}$ stands for the representation dual to $W^1:= W$.
\item generating morphisms by:
\[\vcenter{\hbox{\includegraphics[height=1.2cm, angle=0]{MOY1.pdf}}} \to M_{k,l} \quad\quad \vcenter{\hbox{\includegraphics[height=1.2cm, angle=0]{MOY2.pdf}}} \to M_{k,l}'\]
\end{itemize}
\end{thm}

As a by-product, this theorem says that all computations of Reshetikhin-Turaev invariants of tangles labelled with fundamental representations can be done in the spider category of $\mathfrak{sl}_N$. However, this idea is much older and goes back at least to the paper \cite{MOY} by Murakami, Ohtsuki and Yamada. Thus the morphisms in the $\mathfrak{sl}_N$ spider category are also known as \emph{MOY graphs}.

\begin{rem} In \cite{MOY} tags do not appear and edges labelled by $N$ can be erased completely. Not distinguishing between the two ways of inserting a tag causes an error of a factor of $\pm 1$ on the level of tangle invariants. We adopt this convention and hence work with a smaller category in which the following relation holds:
\begin{equation}
\label{rel7}
\vcenter{\hbox{\includegraphics[height=1.2cm, angle=0]{Rel12.pdf}}}=\vcenter{\hbox{\includegraphics[height=1.2cm, angle=0]{Rel14.pdf}}}
\end{equation}

This new spider category is equivalent to the projective representation category \\$\mathrm{Rep}(U_q(\mathfrak{sl}_N))/{\pm 1}$ and it will suffice for our purposes. 
\end{rem}
\subsection{Constructing Reshetikhin-Turaev tangle invariants}
\label{RT}
We now describe the maps that can be used to define Reshetikhin-Turaev tangle invariants\footnote{The maps we present agree with the original Reshetikhin-Turaev maps up to multiplication by non-zero scalars and powers of $q$.}.

\begin{lem}
\label{doubletag}
Every oriented strand in a web can be inverted at the expense of changing its label $k$ into the complementary label $N-k$ and introducing two tags. 
\end{lem}
\begin{proof}
To see this, we use relation \eqref{rel3} in the projective $\mathfrak{sl}_N$ spider category:
 \[\vcenter{\hbox{\includegraphics[height=1.9cm, angle=0]{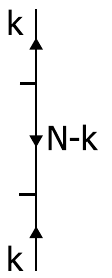}}}=\vcenter{\hbox{\includegraphics[height=1.9cm, angle=0]{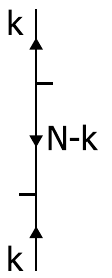}}}=\vcenter{\hbox{\includegraphics[height=1.9cm, angle=0]{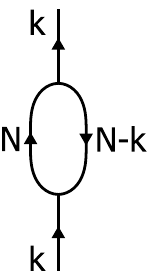}}}= {N-k \brack N-k} \vcenter{\hbox{\includegraphics[height=1.9cm, angle=0]{Rel8.pdf}}}= \vcenter{\hbox{\includegraphics[height=1.9cm, angle=0]{Rel8.pdf}}}\]
\end{proof}

Let $T$ be a tangle in general position labelled by fundamental representations of $U_q(\mathfrak{sl}_N)$. In order to define the Reshetikhin-Turaev invariant of $T$, it suffices to define the maps that correspond to cups, caps and crossings. Via Theorem \ref{ckmthm} we can present these maps in terms of linear combinations of webs.

\begin{defi} (Building blocks for RT invariants) 
To define the maps associated to cups and caps, replace each downward oriented $k$-strand by an upward oriented strand labelled by $N-k$ and place tags on cups and caps, e.g.:
\begin{itemize}
\item For cups colored by $\Lambda^kV$: $\vcenter{\hbox{\includegraphics[height=1.2cm, angle=0]{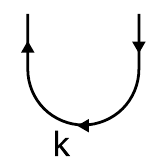}}} \to \vcenter{\hbox{\includegraphics[height=1.2cm, angle=0]{RT2.pdf}}}= M_{k,N-k}'$.
\item For caps colored by $\Lambda^kV$: $\vcenter{\hbox{\includegraphics[height=1.2cm, angle=0]{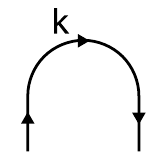}}} \to \vcenter{\hbox{\includegraphics[height=1.2cm, angle=0]{RT4.pdf}}}= M_{k,N-k}$.
\end{itemize}
The maps associated to positive crossings are:
\begin{itemize}
\item $\vcenter{\hbox{\includegraphics[height=2cm, angle=0]{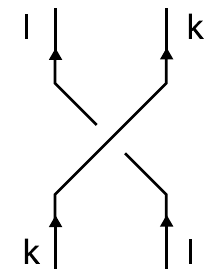}}}\mapsto\quad\sum_{r=0}^l (-1)^{r+(k+1)l} q^{r}\quad \vcenter{\hbox{\includegraphics[height=2cm, angle=0]{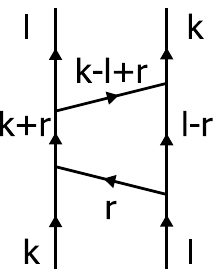}}}$ $\quad$ if $l\leq k$
\item $\vcenter{\hbox{\includegraphics[height=2cm, angle=0]{RT5.pdf}}}\mapsto\quad \sum_{r=0}^k (-1)^{r+(l+1)k} q^{r}\quad\vcenter{\hbox{\includegraphics[height=2cm, angle=0]{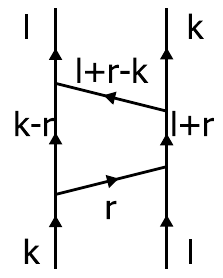}}}$ $\quad$ if $l>k$
\end{itemize}
The formulas for negative crossings are obtained by replacing $q$ by $q^{-1}$. These formulas are taken from \cite{MOY} with the $q$-grading adjusted to our purposes.
\end{defi}

\begin{rem}
For Reshetikhin-Turaev invariants of links colored by fundamental representations of $\mathfrak{sl}_N$ there is an algorithm that does not use any tags at all: without replacing any cups or caps, immediately use the description of the crossing maps to locally replace all crossings by linear combination of webs. The result is a linear combination of closed webs which can then be evaluated to elements in the ground ring. An explicit algorithm for evaluating closed webs that does not use tags is described in Corollary 14.8 and the proof of Theorem 14.7 in \cite{Wu}.
\end{rem}

\begin{rem}
Similar as in the uncolored case, where all involved representations on the strands are standard, there is a stability in the sequence of colored $\mathfrak{sl}_N$ invariants for varying $N$. Using this stability, one can define two-variable colored HOMFLY invariants. However, we have to explain what it means to fix representations on the strands when dealing with $\mathfrak{sl}_N$ invariants of different $N$. Formally, colored HOMFLY link invariants can be defined as invariants of oriented links labelled by Young diagrams, or equivalently by partitions of integers. These partitions specify an irreducible representation as highest weight representation in a tensor product of fundamental representations, independent of $N\gg 0$. In particular, a one-part partition $(k)$ stands for the $k$-th exterior power of the standard representation. To compute the colored HOMFLY invariant of a link pick a $N\gg 0 $, interpret labels $(k)$ as $\Lambda^k V$ and proceed in the same way as when computing the corresponding $\mathfrak{sl}_N$ invariant. However, in doing this treat $a=q^N$ as an independent variable and replace every occurrence of ${N+b \brack c }= \prod_{k=1}^c \frac{[N+b-k+1]}{[k]}= \prod_{k=1}^c \frac{q^{N+b-k+1}-q^{-N-b+k-1}}{q^{k}-q^{-k}} $ by ${b \brack c}_a:= \prod_{k=1}^c \frac{a q^{b-k+1}-a^{-1}q^{-b+k-1}}{q^{k}-q^{-k}}$.
This recipe still produces invariants of framed colored links and it takes values in $\mathbb{C}(q)[a^{\pm 1}]$, but usually not in $\mathbb{C}[a^{\pm 1}, q^{\pm 1}]$. Nevertheless, these invariants are sometimes called \emph{colored HOMFLY polynomials}. Similarly one can define colored HOMFLY invariants of tangles, see section \ref{2ndproof}. 
\end{rem}

\subsection{Web relations via skew Howe duality}
Quantum skew Howe duality and its categorical analogues are powerful tools for studying the (2-)representation category of $U_q(\mathfrak{sl}_N)$. 

Let $X$ be the weight lattice of $\mathfrak{sl}_N$ with simple roots $\alpha_1,\dots, \alpha_{N-1}$ and fundamental weights $\Lambda_1,\dots, \Lambda_{N-1}$ and the non-degenerate bilinear form $\langle . , . \rangle$ given by:
\[\langle \alpha_i, \alpha_j\rangle=\begin{cases}
2 &  \quad\text{if}\quad i=j\\ 
-1 & \quad\text{if}\quad |i-j|=1 \\
0 & \quad\text{else}
\end{cases}.\]
which furthermore satisfies $\langle \alpha_i, \Lambda_j\rangle = \delta_{i j}$. Denote by $Y$ the root lattice of $\mathfrak{sl}_N$ and $Y_\mathbb{C}:=Y\otimes_\mathbb{Z} \mathbb{C}$.

\begin{defi} 
The quantum group $U_q(\mathfrak{sl}_N)$ of $\mathfrak{sl}_N$ is the $\mathbb{C}(q)$ algebra with generators $E_i,F_i,K_i$ for $i=1,\dots, N-1$ and the relations:
\[K_iK_j = K_jK_i, \quad K_jE_iK_j^{-1} = q^{\langle \alpha_i,\alpha_j\rangle},\quad K_jF_iK_j^{-1} = q^{-\langle \alpha_i,\alpha_j\rangle} \]
\[[E_i,F_j] =\delta_{ij} \frac{K_i-K_i^{-1}}{q-q^{-1}}\]
\[[2] E_iE_jE_i =E_i^2E_j + E_jE_i^2 \quad \text{if } |i-j|=1, \quad [E_i,E_j]=0 \quad \text{if } |i-j|>1 \text{ and similarly for Fs.}\]

If $\mu = \sum l_i \Lambda_i \in \mathbb{Z}\langle \Lambda_1,\dots, \Lambda_{N-1}\rangle$ is an integral weight, then we write $K_\mu$ for the product $K_1^{l_1}\cdots K_{N-1}^{l_{N-1}}$. 
\end{defi}

$U_q(\mathfrak{sl}_N)$ is a deformation of the universal enveloping algebra of $\mathfrak{sl}_N$. We now recall the Beilinson-Lusztig-MacPherson idempotent modification $\dot{U}(\mathfrak{sl}_N)$ of $U_q(\mathfrak{sl}_N)$. 
\begin{defi} The idempotent version $\dot{U}(\mathfrak{sl}_N)$ of the quantum group of $\mathfrak{sl}_N$ is defined by adjoining orthogonal idempotents $1_\lambda$ for integral $\mathfrak{sl}_N$ weights $\lambda\in X$ to $U_q(\mathfrak{sl}_N$, subject to the following relations:
\[K_\mu 1_\lambda = 1_\lambda K_\mu = q^{\langle \mu, \lambda\rangle}1_\lambda, \quad E_i 1_\lambda = 1_{\lambda + \alpha_i}E_i,\quad F_i1_\lambda = 1_{\lambda-\alpha_i}F_i\]

Alternatively, $\dot{U}(\mathfrak{sl}_N)$ can be considered as a $\mathbb{C}(q)$-linear category with:
\begin{itemize}
\item objects: integral weights $\lambda \in X$.
\item morphisms: generated by $E_i1_\lambda \in Hom(\lambda,\lambda + \alpha_i )$ and $F_i1_\lambda \in Hom(\lambda,\lambda - \alpha_i )$, where $1_\lambda$ denotes the identity morphism of $\lambda$.
\end{itemize}

Of special importance to categorification is the $\mathbb{Z}[q^{\pm 1}]$-subalgebra $_{\mathcal{A}}\dot{U}(\mathfrak{sl}_N)$ of $\dot{U}(\mathfrak{sl}_N)$, which is generated by divided powers $E_i^{(k)}:= \frac{E_i^k}{[k]!}$ and $F_i^{(k)}:=\frac{F_i^k}{[k]!}$.
 \end{defi}

There exist several variants of the following theorem in the literature. For a more thorough treatment we refer to \cite{Cau2}, \cite{CKM} and \cite{CKL2}. Let $\mathbb{C}_q^s$ and $\mathbb{C}_q^N$ be the standard representations of $U_q(\mathfrak{sl}_{s})$ and $U_q(\mathfrak{sl}_N)$ respectively and denote by $\Lambda$ the appropriate $q$-deformed exterior algebra functor. Then we have:

\begin{thm} (Quantum skew Howe duality) 
The algebra $\Lambda(\mathbb{C}_q^s \otimes \mathbb{C}_q^N)$ carries commuting actions of $U_q(\mathfrak{sl}_{s})$ and $U_q(\mathfrak{sl}_N)$ that constitute a Howe pair. I.e. for any $M\in \mathbb{N}$:
\[End_{U_q(\mathfrak{sl}_N)} \Lambda^M(\mathbb{C}_q^s \otimes \mathbb{C}_q^N) \text{ is generated by } U_q(\mathfrak{sl}_{s})\]
\[ End_{U_q(\mathfrak{sl}_{s})} \Lambda^M(\mathbb{C}_q^s \otimes \mathbb{C}_q^N) \text{ is generated by } U_q(\mathfrak{sl}_N)\]
Moreover, $\Lambda(\mathbb{C}_q^s \otimes \mathbb{C}_q^N)$ splits into $U_q(\mathfrak{sl}_{s})$ weight spaces as
\[\Lambda^M(\mathbb{C}_q^s \otimes \mathbb{C}_q^N)\cong \bigoplus_{i_1+\cdots +i_s=M} \Lambda^{i_1}(\mathbb{C}_q^N)\otimes \cdots \otimes \Lambda^{i_s}(\mathbb{C}_q^N)\] where the summands on the right hand side are weight spaces $V(\lambda)$ associated to weights $\lambda=\sum_{j=1}^{s-1} (i_{j+1}-i_{j}) \alpha_j$ and Chevalley generators act by:

\begin{figure}[h]
\centerline{
\begindc{\commdiag}[40]
\obj(0,1)[1a]{$\cdots \otimes \Lambda^{i_k}(\mathbb{C}_q^N)\otimes \Lambda^{i_{k+1}}(\mathbb{C}_q^N)\otimes \cdots $}
\obj(60,1)[1b]{$ \qquad \cdots \otimes \Lambda^{i_k-1}(\mathbb{C}_q^N)\otimes \Lambda^{i_{k+1}+1}(\mathbb{C}_q^N)\otimes \cdots $}
\obj(20,2)[1c]{$~$}
\obj(40,2)[1d]{$~$}
\obj(20,0)[1e]{$~$}
\obj(40,0)[1f]{$~$}
\mor{1f}{1e}{$F_k$}[\atleft,0]
\mor{1c}{1d}{$E_k$}[\atleft,0]
\enddc
}
\end{figure}
\end{thm}
From the $U_q(\mathfrak{sl}_{s})$ action on $\Lambda^M(\mathbb{C}_q^s \otimes \mathbb{C}_q^N)$ we immediately get a map:
\[{1_{\lambda'}}_{\mathcal{A}}\dot{U}(\mathfrak{sl}_{s})1_{\lambda} \to Hom_{U_q(\mathfrak{sl}_N)}(V(\lambda),V(\lambda'))\]
The weight spaces $V(\lambda)$ are exactly the possible domains and targets for the Reshetikhin-Turaev invariants of tangles with at most $s$ top and bottom boundary points and a labelling by fundamental representations. Furthermore, the basic building blocks of these invariants --- the maps associated to crossings, cups and caps --- are in the image of the above map and hence can be lifted to $_{\mathcal{A}}\dot{U}(\mathfrak{sl}_{s})$. This is due to Cautis, Kamnitzer and Licata \cite{CKL2}.\\

We now provide a dictionary to translate between the two descriptions of maps of $U_q(\mathfrak{sl}_N)$ representations, given by the action of $_{\mathcal{A}}\dot{U}(\mathfrak{sl}_{s})$ on $\Lambda^M(\mathbb{C}_q^s \otimes \mathbb{C}_q^N)$ on the one hand, and by webs on the other hand. In doing this we omit the weight space idempotents $1_\lambda$ and the subscript $r$ of $E_r^{(k)}$ and $F_r^{(k)}$ from the notation. \\

 The action of Chevalley generators $E= E^{(1)}$ and $F= F^{(1)}$ and divided powers $E^{(k)}$ and $F^{(k)}$ can be written in the following way:
\[E^{(k)}\mapsto \vcenter{\hbox{\includegraphics[height=1.2cm, angle=0]{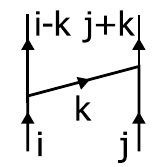}}},\quad F^{(k)}\mapsto \vcenter{\hbox{\includegraphics[height=1.2cm, angle=0]{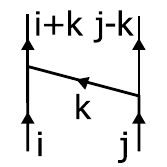}}}\]

The crossing formulas can be written in terms of $E$s and $F$s, e.g. in the case $l\leq k$: \[\vcenter{\hbox{\includegraphics[height=2cm, angle=0]{RT5.pdf}}}\mapsto \quad \sum_{r=0}^l (-1)^{r+(k+1)l} q^{r}\quad \vcenter{\hbox{\includegraphics[height=2cm, angle=0]{UPs.pdf}}}=\sum_{r=0}^l (-1)^{r+(k+1)l} q^{r} E^{(k-l+r)}F^{(r)}\]

Cups and caps can be described as:
\begin{align*}
\vcenter{\hbox{\includegraphics[height=1.2cm, angle=0]{RT4.pdf}}} & \mapsto \vcenter{\hbox{\includegraphics[height=1.2cm, angle=0]{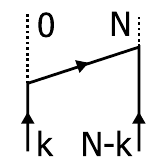}}} = 1_{N}E^{(k)} \cong  1_{-N}F^{(N-k)} =\vcenter{\hbox{\includegraphics[height=1.2cm, angle=0]{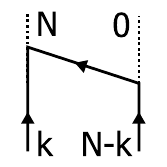}}}\\
\vcenter{\hbox{\includegraphics[height=1.2cm, angle=0]{RT2.pdf}}} &\mapsto \vcenter{\hbox{\includegraphics[height=1.2cm, angle=0]{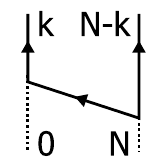}}} = F^{(k)}1_{N} \cong  E^{(N-k)}1_{-N} = \vcenter{\hbox{\includegraphics[height=1.2cm, angle=0]{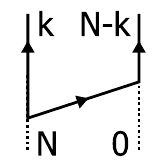}}}
\end{align*}

\begin{wrapfigure}{r}{2.5cm}
  \begin{center}
    \includegraphics[height=3cm]{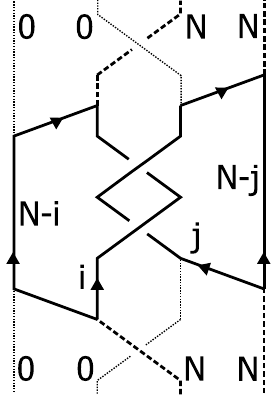}
  \end{center}
\end{wrapfigure} 
Here an issue has to be addressed. The two supposedly isomorphic ways of writing cups and caps as divided power of $E$ or $F$ have different domain or target. This problem is solved by showing that all weight spaces $\Lambda^{i_1}(\mathbb{C}_q^N)\otimes \cdots \otimes \Lambda^{i_s}(\mathbb{C}_q^N)$ are canonically isomorphic to the weight space where all tensor factors $\Lambda^{0}(\mathbb{C}_q^N)$ are permuted to the left and all factors $\Lambda^{N}(\mathbb{C}_q^N)$ are permuted to the right. The isomorphism is realized by braiding the $0$- and $N$-labelled strands (which are indicated by dotted lines) to the outside, using the standard crossing formula. The isomorphism is canonical because it does not depend on how the braiding is realized, as long as all involved crossings have at least a $0$- or $N$-colored strand in them. This is Lemma 7.2 in \cite{Cau2}.%checked!\\

This realization of Reshetikhin-Turaev tangle invariants as images of elements of $_{\mathcal{A}}\dot{U}_q(\mathfrak{sl}_{s})$ under the skew Howe map has the great advantage that many relations necessary for regular isotopy invariance already hold in the quantum group, see \cite{Cau2} Section 7. In fact, quantum skew Howe duality is the main ingredient in the proof of the characterization of the representation category of $U_q(\mathfrak{sl}_N)$ as $\mathfrak{sl}_N$ spider category, as given in \cite{CKM}. 

We summarize the situation:
\begin{thm} (Cautis, \cite{Cau2}, Sections 6 \& 7)
\label{RTviaSHD}
The maps associated to projections of framed oriented tangles colored with irreducible $U_q(\mathfrak{sl}_N)$ representations via the $U_q(\mathfrak{sl}_{s})$ action on the $U_q(\mathfrak{sl}_N)$ representation $\Lambda(\mathbb{C}_q^s\otimes \mathbb{C}_q^N)$ for a suitable $s\in \mathbb{N}$, are invariants of such tangles and agree with the maps defined by Reshetikhin and Turaev up to non-zero scalars and powers of $q$. 
\end{thm}

\subsection{Categorification}
\label{categ}
Simulating skew Howe duality on a categorical level is an attractive approach to defining categorified Reshetikhin-Turaev invariants, because one can get many necessary relations for the invariance proof from the categorified quantum group in a similar way as one gets relations of the $U_q(\mathfrak{sl}_N)$ representation category from the ordinary quantum group. Categorified quantum groups were defined and studied by Lauda in the case of $\mathfrak{sl}_{2}$, see \cite{Lau1}, \cite{Lau2} and the survey \cite{Lau3}, and Khovanov and Lauda, see \cite{KhL1}, \cite{KhL2} and especially \cite{KhL3}, for the $\mathfrak{sl}_N$ case.\\ 

Once one knows categorified quantum groups, it is another problem to define what it means to have an action of a categorified quantum group that categorifies a quantum group representation, e.g. the $U_q(\mathfrak{sl}_{s})$ representation $\Lambda(\mathbb{C}_q^s\otimes \mathbb{C}_q^N)$. Yet another problem is to show that such categorical actions exist at all. We follow Cautis' approach from \cite{Cau1} and \cite{Cau2} where he addresses the latter two problems. He axiomatically defines a what it means to have a $(\mathfrak{g}, \theta)$ action on a 2-category that is the categorical analogue of a $U_q(\mathfrak{g})$ action. This is a fairly minimal definition and because it, a priori, does not mention the full higher representation theoretic structure of the categorified quantum group, existence of such actions is easier to check in practice. However, in \cite{Cau1} Cautis shows that a $(\mathfrak{sl}_{s}, \theta)$ action always extends to a categorical $\mathfrak{sl}_{s}$ action in the stronger sense of \cite{KhL3}. This guarantees that 2-categories with $(\mathfrak{sl}_{s}, \theta)$ actions satisfy a number of further relations which allow the construction of categorical tangle and link invariants via a categorical analogue of skew Howe duality. To keep the exposition compact, we only present Cautis' definition of a $(\mathfrak{sl}_{s}, \theta)$ action and the additional relations necessary for the definition of categorical tangle and link invariants.

We want to mention that there also exists a categorification of skew Howe duality in a strong sense due to Ehrig and Stroppel \cite{ESt}. It consists of commuting 2-actions of quantum groups on derived categories of abelian categories with Koszul self duality intertwining the actions. Furthermore, the involved categories restrict to the ones used in the Mazorchuk-Stroppel \cite{MSt} construction of $\mathfrak{sl}_N$ link homology.  \\

In the following definition we allow the case $s=\infty$, in which we assume that the Lie algebra $\mathfrak{sl}_{\infty}$ has roots and fundamental weights indexed by $\mathbb{Z}$. $X$ still denotes the weight lattice of $\mathfrak{sl}_{s}$.

\begin{defi} (\cite{Cau2}, Section 2.2) %reference checked
A $(\mathfrak{g},\theta)$ action consists of a target graded, additive, $\mathbb{C}$-linear idempotent complete 2-category $\mathfrak{K}$ with:
\begin{enumerate}
\item Objects: indexed by $\lambda \in X$.
\item 1-morphisms: including $E_i 1_\lambda = 1_{\lambda+\alpha_i} E_i$ and $F_i 1_{\lambda+\alpha_i} = 1_\lambda F_i$, where $1_\lambda$ is the identity 1-morphism of $\lambda$.
\item 2-morphisms: including for each $\lambda \in X$ a linear map $Y_\mathbb{C} \rightarrow Hom(1_\lambda, 1_\lambda\langle 2 \rangle)$. 
\end{enumerate}
Here $\langle l\rangle$ is the auto-equivalence given by shifting grading down by $l$.  By abuse of notation we denote by $\theta$ the image of $\theta\in Y_\mathbb{C}$ under such a map.\\
These data are required to satisfy the following relations:
\begin{enumerate}[label=\roman*]
\item $Hom(1_\lambda, 1_\lambda \langle l \rangle)$ is zero if $l < 0$ and one-dimensional if $l=0$ and $1_\lambda \neq 0$. Moreover, the space of maps between any two 1-morphisms is finite dimensional.
\item  $E_i$ and $F_i$ are left and right adjoints of each other up to shifts:
\begin{enumerate}
\item $(E_i 1_\lambda)_R \cong 1_\lambda F_i \langle \langle\lambda, \alpha_i\rangle + 1 \rangle$
\item $(E_i 1_\lambda)_L \cong 1_\lambda F_i \langle - \langle\lambda, \alpha_i\rangle -1 \rangle$.
\end{enumerate}

\item  We have
\begin{align*}
E_i F_i 1_\lambda &\cong F_i E_i 1_{\lambda} \bigoplus_{[\langle\lambda, \alpha_i\rangle]} 1_\lambda \ \ \text{ if } \langle\lambda, \alpha_i\rangle \geq 0 \\
F_i E_i 1_{\lambda} &\cong E_i F_i 1_{\lambda} \bigoplus_{[-\langle\lambda, \alpha_i\rangle]} 1_\lambda \ \ \text{ if } \langle\lambda, \alpha_i\rangle \leq 0
\end{align*}
where $\bigoplus_{p(q)} B$ for a polynomial $p(q)=\sum_ {i\in \mathbb{Z}} a_i q^i \in \mathbb{Z}[q^{\pm 1}]$ means $ \bigoplus_{i\in \mathbb{Z}} (B\langle -i \rangle)^{\oplus a_i}$ and we contract $A\oplus \oplus_{p(q)} B$ to $A\oplus_{p(q)} B$.

\item If $i \neq j$ then $F_j E_i 1_\lambda \cong E_i F_j 1_\lambda$.

\item If $\langle\lambda, \alpha_i\rangle \geq 0$ then map $(I \theta I) \in Hom(E_i 1_\lambda F_i, E_i 1_\lambda F_i\langle 2\rangle)$ induces an isomorphism between $\langle\lambda, \alpha_i\rangle+1$ (resp. zero) of the $\langle\lambda, \alpha_i\rangle+2$ summands $1_{\l+\alpha_i}$ when $\langle \theta, \alpha_i \rangle \neq 0$ (resp. $\langle \theta, \alpha_i \rangle = 0$).\\
If $\langle\lambda, \alpha_i\rangle \leq 0$ then the analogous statement holds for $(I \theta I) \in Hom(F_i 1_\l E_i,F_i 1_\l E_i\langle 2\rangle)$.

\item If $\alpha = \alpha_i$ or $\alpha = \alpha_i + \alpha_j$ for some roots with $\langle \alpha_i,\alpha_j \rangle = -1$ then $1_{\l+r \alpha} = 0$ for $r \gg 0$ or $r \ll 0$.

\item Suppose $i \neq j$ and $\l \in X$. If $1_{\l+\alpha_i}$ and $1_{\l+\alpha_j}$ are nonzero then $1_{\l}$ and $1_{\l+\alpha_i+\alpha_j}$ are also nonzero.
\end{enumerate}
\end{defi}

In \cite{Cau1} Cautis proves that a $(\mathfrak{g},\theta)$ action carries an action of the quiver Hecke algebras and, in the case of $\mathfrak{g}=\mathfrak{sl}_{s}$, such an action extends to a categorical $\mathfrak{sl}_{s}$ action in the sense of \cite{KhL3}. As a consequence, $\mathfrak{K}$ also contains the divided powers $E_i^{(r)}$ or $F_i^{(r)}$, which are adjoint 1-morphisms up to shifts:
\begin{enumerate}[label=\roman*]
\setcounter{enumi}{8}
\item $(E^{(r)}_i 1_\l)_R \cong 1_\l F^{(r)}_i \langle r(\langle\lambda, \alpha_i\rangle+r) \rangle$
\item $(E^{(r)}_i 1_\l)_L \cong 1_\l F^{(r)}_i \langle -r(\langle\lambda, \alpha_i\rangle+r) \rangle$.
\end{enumerate}
The following relations in $\mathfrak{K}$ then follow from Theorem 5.9 and Theorem 5.1 in \cite{KLMS}:
\begin{enumerate}[label=\roman*]
  \setcounter{enumi}{10}
  \item We have
  \[E_i^{(a)}F_i^{(b)}1_\lambda \cong \bigoplus_{j\geq 0}\bigoplus_{{\langle \lambda, \alpha_i\rangle+a-b \brack j}} F_i^{(b-j)}E_i^{(a-j)}1_\lambda\quad \text{if } \langle \lambda, \alpha_i\rangle+a-b\geq 0\]
   \[F_i^{(b)}E_i^{(a)}1_\lambda \cong \bigoplus_{j\geq 0}\bigoplus_{{-\langle \lambda, \alpha_i\rangle-a+b \brack j}} E_i^{(a-j)}F_i^{(b-j)}1_\lambda\quad \text{if } \langle \lambda, \alpha_i\rangle+a-b\leq 0.\]
  
 \item $E^{(a)}E^{(b)} \cong {a+b \brack a } E^{(a+b)}$ and analogously for $F$s.
\end{enumerate}

The definition of categorical knot and tangle invariants is now very similar to the construction of Reshetikhin-Turaev invariants via quantum skew Howe duality. The main ingredient is an $(\mathfrak{sl}_{s},\theta)$ action on a 2-category $\mathfrak{K}$ which is a lift of the $U_q(\mathfrak{sl}_{s})$ action on $\Lambda(\mathbb{C}_q^s\otimes \mathbb{C}_q^N)$. We now explain what this means.

\begin{defi} Let $M$ be a weight module of $U_q(\mathfrak{sl}_{s})$ or equivalently a $\dot{U}(\mathfrak{sl}_{s})$-module. Then we can consider $M$ as a $\mathbb{C}(q)$-linear category with:
\begin{itemize}
\item Objects: weight spaces $M_\lambda= 1_\lambda M 1_\lambda$ indexed by $\lambda\in X$.
\item Morphisms: $Hom(M_\lambda,M_\mu)$ is given by the image of $1_\mu \dot{U}(\mathfrak{sl}_{s}) 1_\lambda$ under the action.
\end{itemize}

Let $\mathfrak{K}$ be a 2-category with a $(\mathfrak{sl}_{s},\theta)$-action. Then we form the Grothendieck category $K(\mathfrak{K})$, which consists of:
\begin{itemize}
\item Objects: the same as $\mathfrak{K}$
\item Morphisms: Split Grothendieck groups of morphism categories of $\mathfrak{K}$, tensored with $\mathbb{C}(q)$. 
\end{itemize}
Here the split Grothendieck group $K(C)$ of an additive, $\mathbb{C}$-linear category $C$ is the free abelian group generated by isomorphism classes of objects of $C$ modulo the relations $[A]=[B_1]+[B_2]$ for all triples $A\cong B_1\oplus B_2$ of objects. If $C$ is graded, $K(C)$ can be regarded as a $\mathbb{Z}[q^{\pm 1}]$-module with the autoequivalence $\langle -1 \rangle$ acting by multiplication by $q$.\\

We say the $(\mathfrak{sl}_{s},\theta)$-action on $\mathfrak{K}$ lifts the $\dot{U}(\mathfrak{sl}_{s})$ action on $M$ if its Grothendieck category is isomorphic to $M$ considered as a $\mathbb{C}(q)$-linear category. In particular, this means that the nonzero objects in $\mathfrak{K}$ are in bijection with the nonzero weight spaces of $M$ and that via $K(.)$ the 1-morphisms $E_i$ and $F_i$ in $\mathfrak{K}$ are sent to the image of $E_i$ and $F_i$ in $End(M)$ under the $\dot{U}(\mathfrak{sl}_{s})$ action. 
\end{defi}

From now on we assume that $\mathfrak{K}$ is a 2-category with a $(\mathfrak{sl}_{s},\theta)$-action that lifts the  $\dot{U}(\mathfrak{sl}_{s})$ action on $\Lambda(\mathbb{C}_q^s\otimes \mathbb{C}_q^N)$. As Cautis explains in \cite{Cau2}, this is sufficient to define categorical link invariants that lift the Reshetikhin-Turaev invariants. For the definition of categorical tangle invariants we further require that the objects of $\mathfrak{K}$ are categories themselves - with 1-morphisms (resp. 2-morphisms) in $\mathfrak{K}$ acting as functors (resp. natural transformations), such that on K-theory we recover $\Lambda(\mathbb{C}_q^s\otimes \mathbb{C}_q^N)$. This means that the Grothendieck groups of the objects $\lambda$, tensored with $\mathbb{C}(q)$, are isomorphic as vector spaces to the weight spaces $\Lambda(\mathbb{C}_q^s\otimes \mathbb{C}_q^N)_\lambda$ and the 1-morphisms in $\mathfrak{K}$ decategorify to the corresponding linear maps between weight spaces that come from the $\dot{U}(\mathfrak{sl}_{s})$ action on $\Lambda(\mathbb{C}_q^s\otimes \mathbb{C}_q^N)$.\\

The categorified invariants considered in this paper live in $Kom^b(\mathfrak{K})$, the bounded homotopy 2-category of $\mathfrak{K}$.

\begin{defi}
Let $\mathfrak{K}$ be as above. Then $Kom^b(\mathfrak{K})$, the bounded homotopy 2-category of $\mathfrak{K}$, consists of:
\begin{itemize}
\item Objects: the same as in $\mathfrak{K}$.
\item 1-morphisms: bounded complexes of 1-morphisms in $\mathfrak{K}$ with differentials built out of 2-morph\-isms in $\mathfrak{K}$.
\item 2-morphisms: chain maps built out of 2-morphisms in $\mathfrak{K}$ between bounded complexes of 1-morphisms, with homotopy equivalent chain maps identified.
\end{itemize} 
$Kom^b(\mathfrak{K})$ is $\mathbb{Z}\oplus \mathbb{Z}$ graded, the second $\mathbb{Z}$ grading being of homological nature. For the homological grading shift, denoted by $[l]$ in \cite{Cau2}, we write the coefficient $t^l$ and and for a $q$-grading shift $\langle k \rangle$ we write $q^{-k}$.
\end{defi}

Given a projection of a framed oriented tangle $T$ colored by irreducible $U_q(\mathfrak{sl}_N)$ representations and supposing that the Reshetikhin-Turaev invariant of $T$ is a map $V(\lambda)\to V(\lambda')$, the categorified invariant is a chain complex with chain spaces built out of 1-morphisms in  $Mor(\lambda, \lambda')$ in $\mathfrak{K}$ and with differentials built out of 2-morphisms in $\mathfrak{K}$. Different tangle projections produce different complexes, however, they are chain homotopy equivalent via chain maps built out of 2-morphisms. \\

We now describe the construction of the complex associated to a tangle projection. Again we prepare the tangle diagram be replacing downward oriented strands by upward oriented strands of complementary labelling and mark all critical points for the height function by tags.

Cups and caps can be described by the 1-morphisms replacing the divided powers in the quantum group:
\begin{align*}
\vcenter{\hbox{\includegraphics[height=1.2cm, angle=0]{RT4.pdf}}} & \mapsto \vcenter{\hbox{\includegraphics[height=1.2cm, angle=0]{SHD1.pdf}}} = 1_{N}E^{(k)} \cong  1_{-N}F^{(N-k)} =\vcenter{\hbox{\includegraphics[height=1.2cm, angle=0]{SHD3.pdf}}}\\
\vcenter{\hbox{\includegraphics[height=1.2cm, angle=0]{RT2.pdf}}} &\mapsto \vcenter{\hbox{\includegraphics[height=1.2cm, angle=0]{SHD2.pdf}}} = F^{(k)}1_{N} \cong  E^{(N-k)}1_{-N} = \vcenter{\hbox{\includegraphics[height=1.2cm, angle=0]{SHD4.pdf}}}
\end{align*}

The main difference to the decategorified setting is that crossings do not get replaced by linear combination of webs, but by complexes:

\[\vcenter{\hbox{\includegraphics[height=2cm, angle=0]{RT5.pdf}}}\mapsto \quad q^{k-l} E^{(k)}F^{(l)} \to \dots \to q^{1} E^{(k-l+1)}F^{(1)} \to \underline{E^{(k-l)}} \quad \text{if } l\leq k\]
\[\vcenter{\hbox{\includegraphics[height=2cm, angle=0]{RT5.pdf}}}\mapsto \quad q^{l-k}F^{(l)}E^{(k)} \to \dots \to q^{1} F^{(k-l+1)}E^{(1)} \to \underline{F^{(l-k)}} \quad \text{if } l> k\]

The underlined terms are in homological grading $0$. The differential in the first complex is given by the following composition of 2-morphisms:

\begin{align*}
q^{r} E^{(k-l+r)}F^{(r)}1_{l-k} \to~q^{k-l+3r-2}E^{(k-l+r-1)} E 1_{l-k-2r} F F^{(r-1)}1_{l-k}&\\
\cong q^{r-1}E^{(k-l+r-1)}E E_R F^{(r-1)} 1_{l-k} &\to q^{r-1}E^{(k-l+r-1)} F^{(r-1)}1_{l-k} 
\end{align*}
Here the first map is inclusion into the lowest $q$-grading summand and the last map is adjunction. The differential in the second complex is similar. In fact, the maps in the differential are unique as grading-preserving maps up to non-zero scalar multiple, see \cite{Cau2} Lemma 4.3. These complexes are versions of the Chuang-Rouquier (or Rickard) complex from \cite{ChR}.\\ %ref checked

As usual for link homology theories, planar composition of diagrams translates into taking formal tensor products of complexes under the replacement described above.\\

Cautis, Kamnitzer and Licata prove in \cite{CKL1} that the complexes associated to positive crossings are invertible in the homotopy category. An inspection of their proof (Proposition 5.4) shows that the inverse is the left (or right) adjoint of this complex. Informally speaking this means interchanging $E$s and $F$s and inverting both the homological and the $q$-grading. The differentials in the resulting complex are again uniquely determined up to non-zero scalar multiple. In order to get invariance under Reidemeister moves of type 2, the inverse complex is used to replace negative crossing. Since taking the mirror image of tangle diagrams or webs also interchanges positive and negative crossings and $E$s and $F$s we have:

\begin{prop}
\label{Mirror} The complexes assigned to a tangle diagram and its mirror image have chain spaces that differ by interchanging $E$s and $F$s and inverting all gradings.
\end{prop}

This matches well with the crossing formulas in the decategorified setting. Finally we get the categorified analogue of Theorem \ref{RTviaSHD}:

%ref checked
\begin{thm} (Cautis, \cite{Cau2} , Proposition 7.9) The complex associated to a framed oriented tangle diagram colored by fundamental representations via categorical skew Howe duality, as described above, is an invariant of the tangle in the homotopy 2-category $Kom^b(\mathfrak{K})$.
\end{thm}
\begin{proof}
The proof is similar to the proof in the decategorified setting via skew Howe duality, see \cite{Cau2} Proposition 7.4.
\end{proof}
%ref checked

\begin{rem} Cautis uses this setup to define categorifications of Reshetikhin-Turaev invariants for labellings by arbitrary irreducible representations. Every irreducible representation $W$ of $U_q(\mathfrak{sl}_N)$ is the highest weight irreducible summand in some $\Lambda^{i_1}(\mathbb{C}_q^N)\otimes\cdots \otimes \Lambda^{i_l}(\mathbb{C}_q^N)$. The main idea is to replace a strand labelled by $W$ by $l$ parallel strands labelled  by $\Lambda^{i_1},\dots, \Lambda^{i_l}$ and to include a highest weight projector, also called a clasp, somewhere along the cable. Following ideas of Rozansky \cite{Roz} Cautis shows that clasps can be realized as infinite twists in the cable. For details see sections 5 and 6 of \cite{Cau2}.%ref checked

In this paper, we are only interested in invariants of links colored by fundamental representations and thus we can get away with working with bounded complexes. In the more general case one allows complexes to be infinite in one direction and makes additional assumptions on the asymptotic behaviour of the $q$-grading. 
\end{rem}

\subsection{Link homology}
\label{simpl}
To get from the categorified colored $\mathfrak{sl}_N$ invariant of a link that lives in a homotopy 2-category to computationally useful link invariants, an additional step is necessary. We recall that the categorified colored $\mathfrak{sl}_N$ invariant of a link $L$ is a complex $\Psi(L)$ of endomorphisms of the highest weight object $\lambda_h$ which corresponds to the sequence $(0,\dots,0, N,\dots,N)$ of labels on the vertical strands under the skew Howe map. Applying the functor $Hom(1_{\lambda_h},.)$ to this complex gives a complex of finite dimensional graded vector spaces with grading preserving differentials whose homology is the required bi-graded link invariant.

In computational practice, however, it is better to find a complex isomorphic to $\Psi(L)$ whose terms are grading shifts of the identity 1-morphism $1_\lambda$. For chain spaces this is equivalent to evaluating a closed web to an element of the ground ring. We demonstrate this in section \ref{conj2}. That the relations in $\mathfrak{K}$, imposed by the existence of a categorical $\mathfrak{sl}_N$ action, are sufficient to simplify the complex follows from the proof of the main theorem in \cite{CKM}, which is called Theorem \ref{ckmthm} here. An example of a computation simplifying  chain spaces and differentials can be found in \cite{Cau2} section 10.

If the categorical 2-representation additionally satisfies that $Hom(1_\lambda,1_\lambda)$ is 1-dimensional and concentrated in $q$-grading $0$, then the resulting link homology theory is completely determined by the defining relations of the categorified quantum group, see \cite{Cau2} section 7.5 and \cite{LQR} section 4.2. Otherwise one can get deformed theories as in \cite{Lee} and \cite{Gor}.

\section{The colored HOMFLY complexes of rational tangles}
\label{rattan}
In this section we use categorical skew Howe duality to compute explicit twist rules that determine the chain spaces of the colored HOMFLY complexes of positive rational tangles labelled by fundamental representations of $\mathfrak{sl}_N$. 

\subsection{Rational tangles}
All two-strand tangles that we consider have boundary points lying on the corners of the unit square in $\mathbb{R}^2$, which we denote by $NE$, $SE$, $SW$ and $NW$. Such tangles can be vertically stacked and horizontally composed by glueing $N*$ to $S*$ and $*E$ to $*W$ labelled boundary points of the respective tangles. 
Consider the special two-strand tangle $T(1,1)$ that is given by a single crossing with the $SW-NE$ strand lying on top. One can act on two-strand tangles by stacking $T(1,1)$ on top or by composing with $T(1,1)$ on the right; we  refer to these operations as \emph{top twist} and \emph{right twist} respectively. The two trivial tangles consisting of the obvious crossing-less matchings in $N-S$ and $W-E$ directions are called $T(0,1)$ and $T(1,0)$.  

The closure of the set containing the trivial tangles $T(0,1)$ and $T(1,0)$ under the operations of top and right twisting is known as the set of \emph{positive rational tangles}. Positive rational tangles are either trivial or can be described by a sequence of positive natural numbers $[a_1,\dots, a_r]$ which describes the construction process: start with $T(1,0)$ and add $a_r$ top twists, then add $a_{r-1}$ right twists, then again $a_{r-2}$ top twists and so on. We can label such a tangle by the rational number 
$\frac{p}{q}>1$ with \[\frac{p}{q} = a_1+\frac{1}{a_2 +\frac{1}{\ddots +\frac{1}{a_r}}}= [a_1,a_2,\dots a_r].\] 

\[\text{E.g. the tangle } T(7,2)\text{ associated to }\frac{7}{2}=[3,2]\text{ is } \quad \vcenter{\hbox{\includegraphics[height=1.0cm, angle=0]{EX72d.pdf}}} \]

More generally, a rational tangle is defined to be a proper embedding of two arcs $\alpha_1$, $\alpha_2$ in a three-ball $B^3$ with ends of the arcs lying on the boundary of $B^3$, such that there is a homeomorphism of pairs:
\[(B^3, \alpha_1\sqcup \alpha_2)\to (D^2\times I, \{x,y\}\times I)\]
 
It is well known that all rational tangles are equivalent (up to isotopy) to a (possibly trivial) rotation of either a positive rational tangle, as described above, or the mirror image of a positive rational tangle. For a proof of this fact see \cite{KaLa}. We restrict our attention to positive rational tangles. \\

\begin{wrapfigure}{r}{2cm}
  \begin{center}
    \includegraphics[height=2cm]{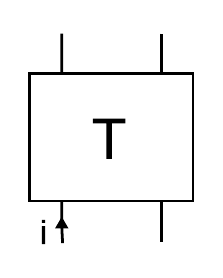}
  \end{center}
\end{wrapfigure} We consider positive rational tangles with additional data. Both arcs of the tangle are equipped with an orientation and a labelling by a fundamental representation $\Lambda^k$ of $U_q(\mathfrak{sl}_N)$ for which we only record the natural number $k$. If we start with an orientation and a labelling on $T(0,1)$ or $T(1,0)$, this induces compatible data on tangles obtained by adding top and right twists. Unless explicitly stated otherwise we assume that the $SW$ boundary point of a tangle is incoming and the corresponding strand is labelled by the color $i$ while the other strand is labelled by a color $j$ with $i\geq j$, see the figure on the right. Clearly this is preserved by adding top or right twists. 

\subsection{Objects in the chain complex}
\label{obj}
We now introduce basic webs that appear as objects in the colored HOMFLY complex of a positive rational tangle.

\begin{defi}
\label{BasicMOY}
The first two rows in Figure \ref{basicGraphs} show our notation for the basic webs that arise in the colored HOMFLY complex of a positive rational tangle labelled by fundamental representations $\Lambda^i$ and $\Lambda^j$ with $i\geq j$. The six variants $UP$, $UPs$, $OP$, $OPs$, $RI$ and $RIs$ correspond to the six possible patterns of boundary data. The first two arguments, e.g. $i$ and $j$ in $UP[i,j,k]$, refer to the colors on the boundary and the third is an index $k$, which we call the \emph{weight} of the web. These webs form bases for the vector spaces of webs with matching boundary data, although we do not immediately use this fact. In intermediate steps we will also use the more general web $UP[i,j,k,l]$ and the operation $.\mapsto (.)^r$ that rotates webs by $\pi$ around the vertical axis, as shown the Figure \ref{basicGraphs}.
\end{defi}
 \begin{figure}[h]
\centerline{
\begindc{\commdiag}[10]
\obj(0,120)[c0]{$UP[i,j,k] =~~ \vcenter{\hbox{\includegraphics[height=2cm, angle=0]{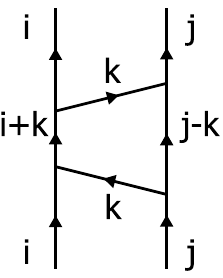}}} $}
\obj(0,60)[c4]{$UPs[i,j,k]=\vcenter{\hbox{\includegraphics[height=2cm, angle=0]{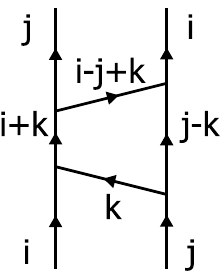}}}$}
\obj(150,120)[c1]{$OP[i,j,k]=~~\vcenter{\hbox{\includegraphics[height=2cm, angle=0]{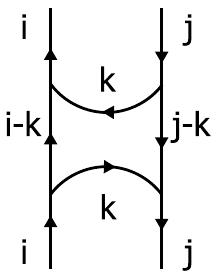}}}$}
\obj(150,60)[c5]{$OPs[i,j,k]=\vcenter{\hbox{\includegraphics[height=2cm, angle=0]{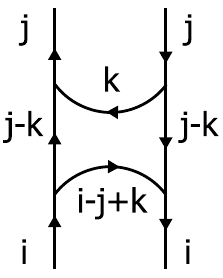}}}$}
\obj(300,120)[c3]{$RI[i,j,k]=\vcenter{\hbox{\includegraphics[height=2cm, angle=0]{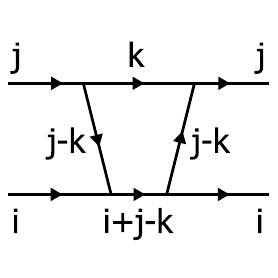}}}$}
\obj(300,60)[c6]{$RIs[i,j,k]=\vcenter{\hbox{\includegraphics[height=2cm, angle=0]{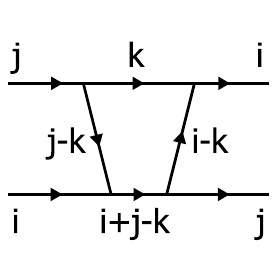}}}$}
\obj(0,0)[c7]{$UP[i,j,k,l]=\vcenter{\hbox{\includegraphics[height=2cm, angle=0]{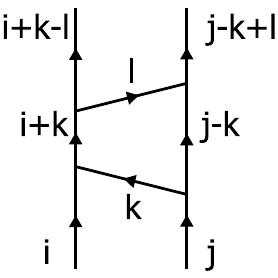}}}$}
\obj(200,0)[c8]{Rotation: $\vcenter{\hbox{\includegraphics[height=2cm, angle=0]{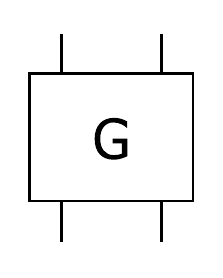}}}\mapsto \vcenter{\hbox{\includegraphics[height=2cm, angle=0]{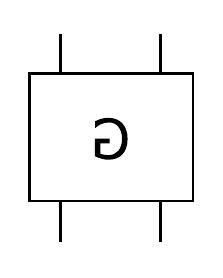}}}$}
\enddc
}
\caption{}
\label{basicGraphs}
\end{figure}~
In Lemma \ref{doubletag} we have seen that two adjacent tags cancel. This shows that the operation of adding a tag close to the end of a strand in a web is an involution and we consider webs that are related by such tag additions as isomorphic. Because we work in the projective setting, it does not matter on which side the tag is placed. Furthermore, relations \eqref{rel4} and \eqref{rel4b} in the spider category say that tags slide past trivalent vertices. In particular we obtain isomorphisms between the special webs introduced above.

\begin{lem} In the projective $U_q(\mathfrak{sl}_N)$ representation category we have:
\label{TagSlideLem}
\begin{align*}
(1)& \quad OPs[i,j,k] \cong UP[N-i,i,i-j+k,k]^r\\
(2)& \quad RI[i,j,k] \cong UP[i,N-i,N-j-i+k,k]\\
(3)& \quad OP[i,j,k] \cong UP[N-j,i,k]^r\\
(4)& \quad RIs[i,j,k]\cong UPs[N-j,i,h]^r\\
(5)& \quad OPs[i,j,k] \cong  RI[N-i,j,k]^r
\end{align*}
\end{lem}
\begin{proof} 
E.g. for (1) we compute:
\[OPs[i,j,k]=\vcenter{\hbox{\includegraphics[height=1.6cm, angle=0]{G0Ps.pdf}}}\cong\vcenter{\hbox{\includegraphics[height=1.6cm, angle=0]{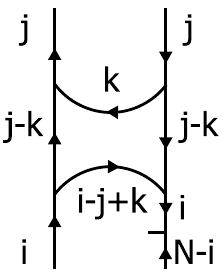}}}=\vcenter{\hbox{\includegraphics[height=1.6cm, angle=0]{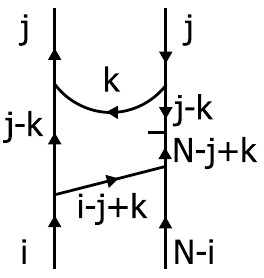}}}=\vcenter{\hbox{\includegraphics[height=1.6cm, angle=0]{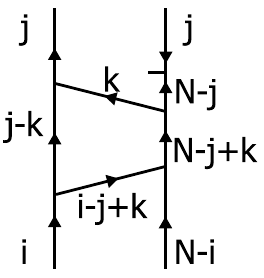}}}\cong \vcenter{\hbox{\includegraphics[height=1.6cm, angle=0]{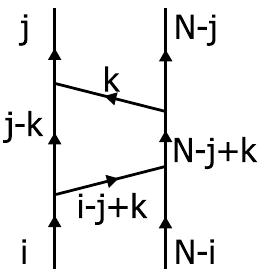}}}= UP[N-i,i,i-j+k,k]^r\]
The other cases are very similar.
\end{proof}

\subsection{Computation of the twist rules}
\label{twistrules}
In order to determine the Poincar\'e polynomial associated to a colored positive rational tangle (or alternatively, the chain spaces in the colored HOMFLY complex), we compute the result of applying a top or right twist to one of the basic webs from Definition \ref{BasicMOY}. It turns out that the dependence of these twist rules on $N\gg 0$ and $i$ for $i\geq j$ can be hidden by introducing the variables $a=q^N$ and $s=q^{i-j}$. Furthermore, we omit global grading shifts in the formulation of the twist rules.

\begin{defi} 
${h \brack k}$ is a Laurent polynomial in $q$ that is supported in degrees $-k(h-k)$ up to $k(h-k)$. We thus define the variants:
\[{h \brack k}^+ := q^{k(h-k)}{h \brack k} \quad \text{and}\quad {h \brack k}^- := q^{-k(h-k)}{h \brack k}\]
which have lowest and highest degree $0$ respectively.
\end{defi} 	
 	
The following proposition relies heavily on results from section 4 in \cite{Cau2} that describe how (generalizations of) crossing complexes can absorb $E$s and $F$s. To be more precise, the result of composing a crossing complex with a divided power 1-morphism is a complex $\tau$ or $\tau'$ which again has chain spaces that are products of $E$s and $F$s and $q$-grading preserving differentials that are uniquely determined up to non-zero scalar. However, in general they are not invertible in the homotopy 2-category. 
\begin{prop} 
%% be consistent with \cong and =!!!!
Adding top twists to $UP[i,j,k]$, $UPs[i,j,k]$ and $OPs[i,j,k]$ has the following effect on the level of Poincar\'e polynomials:
\comm{\begin{enumerate}
\item $TUP[i,j,k]\cong \sum_{h=k}^j t^{-h} s^{k} q^{k^2+h} {h \brack k}^+ UPs[i,j,h] $
\item $TUPs[i,j,k]\cong \sum_{h=k}^j t^{-h} s^{h} q^{k^2+h} {h \brack k}^+ UP[i,j,h] $
\item $TOPs[i,j,k]\cong \sum_{h=k}^j t^{-h} a^k s^{h-k} q^{k(k-2j)+h} {h \brack k}^+ RI[i,j,h] $
\end{enumerate}}
\begin{align*}
(1)& \quad TUP[i,j,k]\cong \sum_{h=k}^j t^{-h} s^{k} q^{k^2+h} {h \brack k}^+ UPs[i,j,h]\\
(2)& \quad TUPs[i,j,k]\cong \sum_{h=k}^j t^{-h} s^{h} q^{k^2+h} {h \brack k}^+ UP[i,j,h]\\
(3)& \quad TOPs[i,j,k]\cong \sum_{h=k}^j t^{-h} a^k s^{h-k} q^{k(k-2j)+h} {h \brack k}^+ RI[i,j,h]
\end{align*}
\end{prop}
\begin{proof} For the first isomorphism we just translate into Cautis' notation and apply Proposition 4.5 from \cite{Cau2}.

\begin{align*}
TUP[i,j,k] &= 1_{i-j} \tau_{i-j} 1_{j-i} E^{(k)}F^{(k)} \cong t^{-k} q^{k(i-j+k+1)}1_{i-j} \tau_{i-j+k} 1_{j-i-2k} F^{(k)} \\
&\cong \sum_{r=0}^{j-k} t^{-(k+r)} q^{k(i-j+k+1) + r(k+1)}1_{i-j} E^{(i-j+k+r)}F^{(r)}F^{(k)}1_{j-i} 
\end{align*}
After replacing $F^{(r)}F^{(k)}$ by ${r+k \brack k}F^{(r+k)} = q^{-k r}{r+k \brack k}^+F^{(r+k)}$ via relation xii and re-parametrising the summation, this yields (1).
\begin{align*} TUPs[i,j,k] &=  1_{j-i} \tau'_{i-j} 1_{i-j} E^{(i-j+k)}F^{(k)}  \cong  1_{j-i} \tau'_{i-j} 1_{i-j}F^{(k)} 1_{i-j+2k} E^{(i-j+k)} \\
&\cong t^{-k} q^{k(i-j+k+1)} 1_{j-i} E^{(k)}  1_{j-i-2k} \tau'_{i-j+2k}  1_{i-j+2k} E^{(i-j+k)} \\
&\cong t^{-k} q^{k(i-j+k+1)} 1_{j-i} E^{(k)}  1_{j-i-2k} \tau'_{k}  1_{i-j}  \\
&= \sum_{r=0}^{j-k} t^{-(k+r)} q^{(k+r)(i-j+k+1)} 1_{j-i} E^{(k)}E^{(r)}F^{(k+r)}    1_{i-j}  \\
&\cong \sum_{h=k}^{j} t^{-h} q^{h(i-j+k+1)} {h \brack k}1_{j-i} E^{(h)}F^{(h)}    1_{i-j} \\ &= \sum_{h=k}^j t^{-h} q^{h(i-j)} q^{k^2+h} {h \brack k}^+ E^{(h)}F^{(h)} 1_{i-j} 
\end{align*}
which is (2). Here we have used relations xi, Corollary 4.6 and Proposition 4.5 in \cite{Cau2} and relation xii.

For (3) we first use isomorphism (1) from Lemma \ref{TagSlideLem} together with the fact that tags slide through crossing complexes. Then we apply Corollary 4.6 and Proposition 4.5 from \cite{Cau2} and finally isomorphism (2) from Lemma \ref{TagSlideLem} and relation xii.
\begin{align*}
TOPs[i,j,k] &\cong TUP[N-i,i,i-j+k,k]^r = 1_{2j-N} \tau_{N-2j}' F^{(k)} E^{(i-j+k)} 1_{N-2i}\\
 &\cong t^{-k}q^{k(N-2j+k+1)}1_{2j-N} E^{(k)} 1_{2j-N-2k} \tau_{N-2j+2k}'  E^{(i-j+k)} 1_{N-2i}\\ 
 &\cong t^{-k}q^{k(N-2j+k+1)}1_{2j-N} E^{(k)} 1_{2j-N-2k} \tau_{N-j-i+k}'  1_{N-2i}\\
  &= \sum_{r=0}^{j-k} t^{-(k+r)}q^{k(N-2j+k+1) +r(i-j+k+1) }1_{2j-N} E^{(k)}E^{(r)}F^{(N-j-i+k+r)}  1_{N-2i}\\
  &=\sum_{h=k}^j t^{-h} a^k q^{(h-k)(i-j)} q^{k(k-2j)+h} {h \brack k}^+1_{2j-N} E^{(h)}F^{(N-j-i+h)}  1_{N-2i}\\
   &=\sum_{h=k}^j t^{-h} a^k q^{(h-k)(i-j)} q^{k(k-2j)+h} {h \brack k}^+ UP[i,N-i,N-j-i+h,h]\\
    &=\sum_{h=k}^j t^{-h} a^k q^{(h-k)(i-j)} q^{k(k-2j)+h} {h \brack k}^+ RI[i,j,h].
\end{align*}
\end{proof}

\begin{cor} Adding top twists to $OP[i,j,k]$, $RI[i,j,k]$ and $RIs[i,j,k]$ has the following effect on the level of Poincar\'e polynomials:
\begin{align*}
(1)& \quad TOP[i,j,k]\cong \sum_{h=k}^j t^{-h} a^k s^{-k} q^{k(k-2j)+h} {h \brack k}^+ RIs[i,j,h] \\
(2)& \quad TRI[i,j,k]\cong \sum_{h=k}^j t^{-h} a^h s^{k-h} q^{k^2+h(1-2j)} {h \brack k}^+ OPs[i,j,h]\\
(3)& \quad TRIs[i,j,k]\cong \sum_{h=k}^j t^{-h} a^h s^{-h} q^{k^2+h(1-2j)} {h \brack k}^+ OP[i,j,h]
\end{align*}
\end{cor}
\begin{proof} For (1):
\begin{align*}
TOP[i,j,k] \cong TUP[N-j,i,k]^r &\cong \sum_{h=k}^i t^{-h} q^{k(N-i-j)} q^{k^2+h} {h \brack k}^+ UPs[N-j,i,h]^r\\
&\cong \sum_{h=k}^i t^{-h} a^k q^{-k(i-j)} q^{k(k-2j)+h} {h \brack k}^+ RIs[i,j,h]
\end{align*}
Here we have used isomorphism (3) and (4) from Lemma \ref{TagSlideLem} and the rule for $TUP$. The sum gets truncated because $ UPs[N-j,i,h]^r \cong 0\cong RIs[i,j,h]$ for $h>j$. 

For (2):
\begin{align*}
TRI[i,j,k] \cong TOPs[N-i,j,k]^r &= \sum_{h=k}^j t^{-h} a^k q^{(h-k)(N-i-j)} q^{k(k-2j)+h} {h \brack k}^+ RI[N-i,j,h]^r \\
&= \sum_{h=k}^j t^{-h} a^h q^{(k-h)(i-j)} q^{k^2+h(1-2j)} {h \brack k}^+ RI[N-i,j,h]^r \\
&\cong \sum_{h=k}^j t^{-h} a^h q^{(k-h)(i-j)} q^{k^2+h(1-2j)} {h \brack k}^+ OPs[i,j,h] 
\end{align*}
Here we have used isomorphism (5) from Lemma \ref{TagSlideLem} and the rule for $TOPs$.

For (3):
\begin{align*}
TRIs[i,j,k] \cong TUPs[N-j,i,k]^r &\cong 
\sum_{h=k}^i t^{-h} q^{h(N-j-i)} q^{k^2+h} {h \brack k}^+ UP[N-j,i,h]^r\\
&\cong \sum_{h=k}^i t^{-h} a^h q^{-h(i-j)} q^{k^2+h(1-2j)} {h \brack k}^+ OP[i,j,h]\end{align*}
where we have used isomorphisms (4) and (3) from Lemma \ref{TagSlideLem} and the rule for $TUPs$. The sum gets truncated because $UP[N-j,i,h]^r\cong 0 \cong  OP[i,j,h]$ for $h>j$.
\end{proof}

\begin{cor} We compute the right twist rules by reflecting top twist rules.
\begin{align*}
(1)& \quad RUP[i,j,k]\cong\sum_{h=0}^k t^{-h} a^h s^{k-h}q^{k(2j-k) + h(1-2j)} {j-h \brack k-h}^- OP[i,j,h]\\
(2)& \quad RUPs[i,j,k]\cong\sum_{h=0}^k t^{-h} a^h s^{-h}q^{k(2j-k) + h(1-2j)} {j-h \brack k-h}^- OPs[i,j,h]\\
(3)& \quad ROP[i,j,k]\cong\sum_{h=0}^k t^{-h} a^k s^{h-k}q^{-k^2+h} {j-h \brack k-h}^- UP[i,j,h]\\
(4)& \quad ROPs[i,j,k]\cong\sum_{h=0}^k t^{-h} a^k s^{-k}q^{-k^2+h} {j-h \brack k-h}^- UPs[i,j,h]\\
(5)& \quad RRI[i,j,k]\cong\sum_{h=0}^k t^{-h} s^{k}q^{-k(k-2j)+h} {j-h \brack k-h}^- RIs[i,j,h]\\
(6)& \quad RRIs[i,j,k]\cong\sum_{h=0}^k t^{-h} s^{h}q^{-k(k-2j)+h} {j-h \brack k-h}^- RI[i,j,h]
\end{align*}
\end{cor}
 \begin{proof}
	Note that after reflecting across the plane spanned by the $SW-NE$ diagonal and the normal to the blackboard, the problem of adding a right twist transforms into adding a top twist to the reflected web. Proposition \ref{Mirror} extends to the case of complexes associated to knotted webs and thus under reflection $t$, $q$, $a$  and $s$ get replaced by their inverses and $UP(s)[i,j,k]$ and $RI(s)[i,j,j-k]$ as well as $OP[i,j,k]$ and $OPs[i,j,j-k]$ are interchanged.
 	\end{proof}
 	
\begin{rem}
An essential feature of these rules is that their dependence on the rank $N$ and the higher color $i$ can be hidden by introducing $a=q^N$ and $s=q^{i-j}$. This proves Theorem \ref{thmA}.
\end{rem}

\section{A geometric model for the colored HOMFLY complex of a rational tangle}
\label{geom}
In this section we introduce a geometric algorithm for computing the chain spaces in the colored HOMFLY complex of a positive rational tangle $T(p,q)$ labelled with fundamental $\mathfrak{sl}_N$ representations $\Lambda^i$ and $\Lambda^j$ with $i\geq j$.

\subsection{The geometric setup}
We start by drawing a picture that resembles a 2-bridge diagram of the denominator closure of the tangle. Let $\sigma\in\{1,-1\}$ be the parity of the length of the continued fraction expansion of $\frac{p}{q}$ and $par(p)\in\{1,-1\}$ the parity of $p$.

Draw the intervals $[-2,-1]$ and $[1,2]$ on the real axis in $\mathbb{C}$ on a piece of paper, partition them into $p$ parts of equal size and mark the divisions between the parts by small vertical line segments. For both divided intervals, number the $2p$ endpoints of line segments by $0,\dots, 2p-1\in \frac{\mathbb{Z}}{(2p)}$, starting from the points $1$ and $-1$ respectively and proceeding clockwise or anticlockwise, depending on $\sigma=\mp 1$. 
Next, draw an arc $\alpha$ starting at the $0$-labelled point $x=-par(p)\sigma$, then proceeding to the point labelled $q$ on the other interval, intersecting the interval transversely, proceeding to the point labelled $2q$ and so on until it hits the set $\{-2,-1,1,2\}$ again. We require $\alpha$ to have no self intersections and the minimal possible number of intersections with the two intervals. Further, we fix the picture uniquely up to isotopy by requiring that $\alpha$ has no intersections with the interval $(-\infty,-2]$ on the real axis if $\sigma=-1$ and no intersections with the interval $[2,\infty)$ on the real axis if $\sigma=1$.

Next, we draw the imaginary axis labelled $l_p$ and a parallel of it labelled $l_q$ that intersects one of the intervals $[1,2]$ and $[-2,-1]$ on the real axis, depending on $\sigma=\mp 1$, and has the minimal possible number of intersections with the arc $\alpha$. The two vertical lines $l_p$ and $l_q$ then have $p$ and $q$ intersections with $\alpha$ respectively. 
We call these intersections with the left or right vertical \emph{left} and \emph{right primary intersections}. The figure below shows the case $\frac{p}{q}=\frac{5}{2}$ and how this picture can be interpreted as a projection of the tangle. 

\[\vcenter{\hbox{\includegraphics[height=3cm, angle=0]{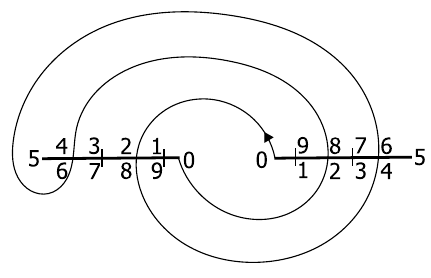}}}\quad\to\quad \vcenter{\hbox{\includegraphics[height=3cm, angle=0]{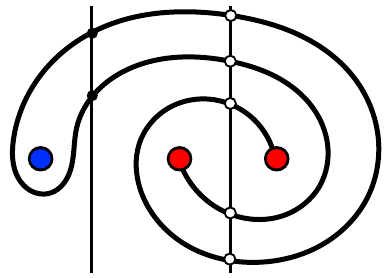}}}\quad = \quad \vcenter{\hbox{\includegraphics[height=3cm, angle=0]{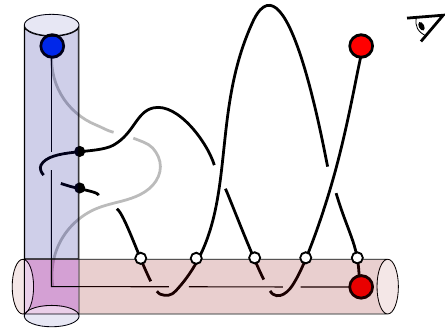}}}\]

The diagrams we draw are essentially 2-bridge diagrams of the denominator closure of the positive rational tangle, with one over-bridge erased. We distinguish the cases according to $\sigma$ and $ par(p)$ and choose a more complicated way of relating the drawing to the rational $\frac{p}{q}$ in order to get a simpler relation between the drawing and the continued fraction expansion of $\frac{p}{q}$. We explain this in the following lemma, whose proof is left to the reader.

\begin{lem} The picture can inductively be constructed in the following way. Start with the trivial diagram. For a top twist, bend the left vertical towards the right one, then flip over the left hand side of the diagram into the middle to straighten the left vertical again. For a right twist, do the analogue for the right vertical. The trivial diagram and a top twist are illustrated below.

\[\text{trivial:}\vcenter{\hbox{\includegraphics[height=1.4cm, angle=0]{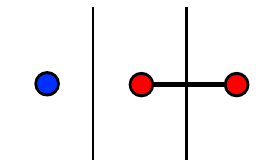}}},\quad\text{top twist:}\vcenter{\hbox{\includegraphics[height=1.4cm, angle=0]{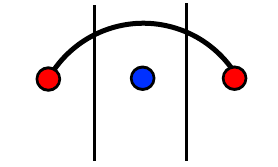}}}  \to \vcenter{\hbox{\includegraphics[height=1.4cm, angle=0]{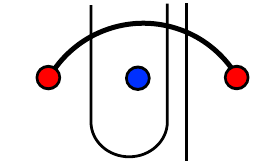}}}=\vcenter{\hbox{\includegraphics[height=1.4cm, angle=0]{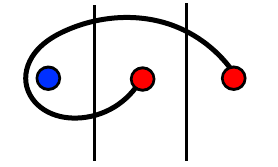}}} \]
\end{lem}

\begin{defi}
We begin the computation of the colored HOMFLY complex by starting with $UP[i,j,0]$ or $OP[i,j,0]$, depending on the orientation on the tangle, and then applying a sequence of top and right twists.  In these two cases we label the three distinguished points by $X^+$, $X^-$ and $Y$ as in the following figures.
\[ \text{Start configuration for UP: }  \vcenter{\hbox{\includegraphics[height=1.4cm, angle=0]{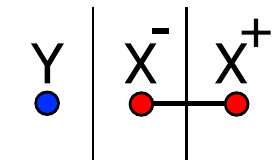}}}\quad \text{Start configuration for OP: } \vcenter{\hbox{\includegraphics[height=1.4cm, angle=0]{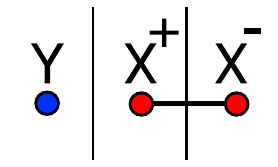}}} \]
\end{defi}

\begin{cor} Via the interpretation of top and right twisting as bending verticals, the six permutations of the labels $X^+ , X^-$ and $Y$ that arise in tangle diagrams correspond to the six types of webs introduced above. This is shown in Figure \ref{configs}.
\end{cor}

 \begin{figure}[h]
\centerline{
\begindc{\commdiag}[40]
\obj(45,28)[1a]{$UP\vcenter{\hbox{\includegraphics[height=1.4cm, angle=0]{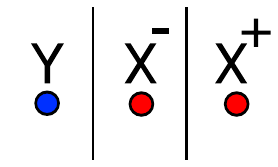}}}$}
\obj(75,28)[1b]{$UPs\vcenter{\hbox{\includegraphics[height=1.4cm, angle=0]{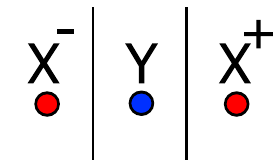}}}$}
\obj(20,14)[2a]{$OP \vcenter{\hbox{\includegraphics[height=1.4cm, angle=0]{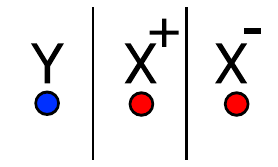}}}$}
\obj(45,0)[2b]{$RIs\vcenter{\hbox{\includegraphics[height=1.4cm, angle=0]{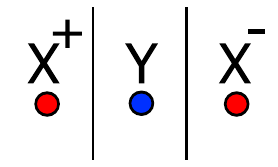}}}$}
\obj(75,0)[2c]{$RI\vcenter{\hbox{\includegraphics[height=1.4cm, angle=0]{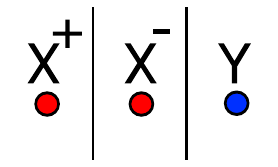}}}$}
\obj(100,14)[2d]{$OPs\vcenter{\hbox{\includegraphics[height=1.4cm, angle=0]{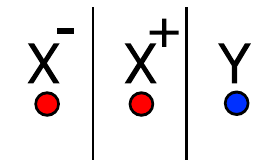}}}$}
\mor{1a}{1b}{$T$}[\atleft,\doubleopposite]
\mor{1a}{2a}{$R$}[\atright,\doubleopposite]
\mor{2a}{2b}{$T$}[\atright,\doubleopposite]
\mor{2b}{2c}{$R$}[\atright,\doubleopposite]
\mor{2c}{2d}{$T$}[\atright,\doubleopposite]
\mor{1b}{2d}{$R$}[\atleft,\doubleopposite]
\enddc
}
\caption{}
\label{configs}
\end{figure}

\subsection{The main theorem}

\begin{defi}
In order to find the generators of weight $h$ for the colored HOMFLY complex with respect to colors $i$ and $j$ in the geometric picture, we replace the left vertical by $h$ parallel copies and the right vertical by $j-h$ parallel copies. We call the first kind of parallels, $w_1,\dots, w_h$, the \emph{weighted verticals} and the others, $u_1,\dots u_{j-h}$, \emph{unweighted verticals}. 
\end{defi}

\begin{thm}
\label{mainthm}
The generators of weight $h$ in the colored HOMFLY complex with respect to colors $i$ and $j$ are in bijection with $j$-tuples \[(x_1,\dots, x_j)\in (w_1\cap \alpha)\times \cdots \times (w_h\cap \alpha)\times (u_1\cap \alpha)\times \cdots \times (u_{j-h}\cap \alpha)\] of intersections points of verticals with $\alpha$. To find the relative gradings of these generators, it is sufficient to determine:
\begin{enumerate}
\item  The grading difference of two generators with equal weight that differ only in one coordinate. 
\item For each $h$, the grading difference of some special generators of weights $h$ and $h-1$.
\end{enumerate}
By iteration rule (1) determines the relative gradings of all generators of the same weight. Rule (2) then fixes the relative gradings between weight groups. \\

Ad (1): The following determines the grading difference between two generators of weight $h$ that differ only in one coordinate. Let $\overline{x}=(x_1,\dots,x_k,\dots, x_j)$ and $\overline{x}'=(x_1,\dots,x_k',\dots, x_j)$ be the two generators with $x_k\neq x_k'$.
Let $\alpha'$ be the segment of $\alpha$ starting at $x_k$ and ending at $x_k'$, $\beta$ the segment of the $k^{th}$ vertical starting at $x_k$ and ending at $x_k'$ and $D$ the domain enclosed by $\alpha'$ and $\beta$ in $\mathbb{R}^2$. $D$ is a singular $2$-chain, i.e. a formal sum of closed discs with multiplicities in $\mathbb{Z}$ determined by the winding of $-\beta\alpha'$ around interior points. In particular, $D$ can be written as a $\mathbb{Z}$-linear combination of \emph{simple discs} which we define to be the closures of bounded path-components of the complement of the union of $\alpha$ with one vertical, equipped with the standard orientation.  The grading difference between $\overline{x}$ and $\overline{x}'$ is the product of two terms:
\begin{itemize}
\item (Additive part $Q$) The graded intersection number of $D$ with the set of distinguished points $\{X^+, X^-, Y\}$ where intersections with a simple disc $D_s$ count as:
\[D_s\cdot X^+=\begin{cases}
\frac{a^2}{t q^{4j-2} s^2}&\text{ if}~ X^+\in D_s\\ 
1 &\text{ if}~ X^+\notin D_s \\
\end{cases}, ~ D_s\cdot X^- = \begin{cases}
\frac{q^2 s^2}{t}&\text{ if}~ X^-\in D_s\\ 
1 &\text{ if}~ X^-\notin D_s \\
\end{cases}, ~ D_s\cdot Y = \begin{cases}
\frac{q^2}{t}&\text{ if}~ Y\in D_s\\ 
1 &\text{ if}~ Y\notin D_s \\
\end{cases}\]
\item (Non-additive part $q-Q$) The graded intersection number of $D$ with the other intersection points $\{x_1,\dots, x_{k-1},x_{k+1},\dots x_j\}$, with a simple disc $D_s$ contributing:
\[D_s\cdot x_r = \begin{cases}
q^4 &  \quad\text{if}\quad x_r\in D_s^\circ\\ 
q^2 & \quad\text{if}\quad x_r \in \partial D_s \\
1 & \quad\text{else}
\end{cases}\]
\end{itemize}

Ad (2): The following determines the grading shift between two generators of weight $h$ and $h-1$ that agree outside the $h^{th}$ coordinate and whose $h^{th}$ coordinates are related by sliding the $h^{th}$ vertical (which is the innermost vertical) from the weighted to the unweighted side. We can think of this sliding as geometric realization of the differential tying together the two generators. The grading of this differential depends on which distinguished point the vertical crosses while sliding:
\[\text{Vertical sliding across } \begin{cases}
Y\\ 
X^-\\
X^+\\
\end{cases} \text{ as in }
\begin{cases}
UPs\text{ and }RIs\\ 
UP\text{ and }RI \\
OP\text{ and }OPs\\
\end{cases}\text{ causes shift by } \begin{cases}
\frac{t}{q}\\ 
\frac{t}{q s}\\
\frac{t q^{2j-1} s}{a}\\
\end{cases}\]

If two generators are related by sliding over the innermost vertical, we say that the generators are related by a \emph{simple slide}.

\end{thm}

The proof of this theorem occupies the rest of this section and is split into several lemmata. Then Theorem \ref{thmB} follows immediately from the fact that the $a$-, $s$-, $t$- and $Q$-gradings are additive.

\begin{lem}
\label{welldefgrad}
The grading differences stated in the theorem define a $\mathbb{Z}^4$ grading on generators with shifts denoted by powers of $a$, $q$, $t$ and $s$. 
\end{lem}
\begin{proof} While rule (2) does not cause any ambiguities, there are several ways to compute the grading difference of two generators of equal weight by repeated application of rule (1). We need to check that all of these ways yield the same result. This is easily seen for the additive part. It remains to check the non-additive part.\\

Consider first the case of generators that differ in exactly two coordinates, say $\overline{x}=(x_1,x_2,\dots)$ and $\overline{y}=(y_1,y_2,\dots)$, where the two domains $D_1$ and $D_2$ connecting $x_1$ to $y_1$ and $x_2$ to $y_2$ respectively are simple discs, possibly with the opposite orientation. We need to check that the non-additive part of the grading that we get from $D_2\circ D_1\colon \overline{x}\to (y_1,x_2,\dots) \to \overline{y}$ is the same as the one from $D_1\circ D_2\colon \overline{x}\to (x_1,y_2,\dots) \to \overline{y}$. It is sufficient to consider the non-additive component that comes from intersections with the first two coordinates, since all other contributions will be equal a priori. 
There are three cases to consider:
\begin{enumerate}
\item $D_1 \cap D_2 = \emptyset$, the trivial case.
\item $D_1$ and $D_2$ intersect and have opposite orientations. In this case, one way of composing $D_1$ and $D_2$ and intersecting with the first two coordinates produces no non-additive grading shifts and the other way produces two cancelling grading shifts. The cancelling shifts are $q^{\pm 2}$ or $q^{\pm 4}$, depending on whether the discs intersect only along their boundary or in their interior.
\item $D_1$ and $D_2$ intersect and have matching orientations. Here the two ways of composing $D_1$ and $D_2$ produce non-additive grading shifts at different intersection points, but of equal value  $q^{\pm 2}$ or $q^{\pm 4}$, depending on whether the discs intersect only along their boundary or in their interior.
\end{enumerate}

In the general case of weight $h$-generators that might differ in several coordinates, every way of getting from one generator to the other by transforming one coordinate at a time can be described as a sequence of changes via simple discs. By the previous argument, we can permute these changes and cancel redundant pairs of inverse changes without altering the grading shift. This way one can reach a unique reduced expression which consists of the minimal number of changes via simple discs and is ordered by the indices of the coordinates in which the changes take place.  
\end{proof}

It is clear that our picture-way of calculating the chain spaces in the colored HOMFLY complex produces only one generator in the case of the trivial tangle. With this as the start of an induction proof it suffices to show that the geometric algorithm observes the twist rules from section \ref{twistrules}.\\

To prepare some notation for the following definition, we look at a top twist applied to a diagram $P_1,$ producing a diagram $P_2$.  If we only draw the verticals $l_p$ and $l_q$ for the moment, we can distinguish three sets of primary intersections: $L_1$, the intersections with the left vertical in $P_1$, $R$, the intersections with the right vertical in $P_1$ and $L_2= L_1\sqcup R'$ the set of intersections with the left vertical in $P_2$. Note that $R'$ is in bijection with $R$ and $R$ also labels the intersections with the right vertical in $P_2$. 

Generators $(x_1,\dots, x_j)\in (w_1\cap \alpha)\times \cdots \times (w_h\cap \alpha)\times (u_1\cap \alpha)\times \cdots \times (u_{j-h}\cap \alpha)$ of the colored HOMFLY complex have an equivalent description by $(x_1',\dots, x_j')$ with $x_1',\dots, x_h'$ being left primary intersections and $x_{h+1}',\dots, x_j'$ right primary intersections, where the order remembers the order in which these primary intersections are placed on parallels of the two verticals.

\begin{defi} (Top twist case)
Let $\overline{x}=(x_1',\dots, x_j')$ be a generator of the colored HOMFLY complex of $P_2$, which is obtained from $P_1$ by a top twist. Partition the left primary intersections $x_1',\dots, x_h'$ into two subsequences according to whether they belong to $L_1$ or $R'$. Next reverse the subsequence in $L_1$ to get $l$, concatenate it with the sequence $r_1$ in $R$ corresponding to the sequence in $R'$ and finally append the sequence $r_2= x_{h+1}',\dots, x_j'$. The sequence $l\cdot r_1\cdot r_2$ represents a generator for the complex of $P_1$ and we call it the \emph{parent} of its \emph{child} $\overline{x}$. The definition of the relation between child and parent is analogous in the case of a right twist.
\end{defi}

One can think of the top twisting process as bending over the left verticals towards the right and allowing them to steal verticals and intersection points from the right verticals. Similarly, in a right twist, the right verticals bend over to the left and steal verticals and intersection points from the left verticals. The twist rules describe how a generator --- the parent --- in $P_1$ gives rise to several generators --- its children --- in $P_2$.\\

The relative gradings of generators in $P_2$  are influenced by three factors in the geometric picture, which we will treat in this order:
\begin{enumerate}
\item How the new verticals that are stolen from the other side get sorted in between the old verticals. This corresponds to the quantum binomial coefficients in the twist rules.
\item How many verticals are stolen. This corresponds to the coefficients depending on $h-k$ in the twist rules.
\item The weight of the parent in $P_1$. This corresponds to coefficients depending on $k$ in the twist rules.
\end{enumerate}

\begin{lem}
\label{sortlem} The generators corresponding to the possible ways of sorting $h-k$ stolen verticals (with intersections) in between $k$ existing verticals have relative $q$-gradings as described by the quantum binomial coefficient ${h \brack k}$. In the picture with bent verticals, the lowest $q$-grading configuration has all intersection points as far right as possible and the highest $q$-grading configuration has all intersection points as far left as possible.
\end{lem} 
\begin{proof}
We first look at the relevant region where the weighted verticals are bent over. The figure below shows a typical situation. If we assume that it is the result of a top twist, then we see a generator represented by a tuple $(\dots,A,B,C,D,\dots)$ on the left hand side. 

To prove the statement of the lemma, it suffices to show that if an intersection $B$ on the left of the bend and an intersection $C$ on the right of the bend, lying on adjacent verticals, swap verticals such that $C$ moves left, this causes a grading shift of $q^2$. The result of such a swap is shown on the right hand side of the figure. If it is the result of a top twist, the generator is represented by the tuple $(\dots, A,C,B,D,\dots)$. 

\[\vcenter{\hbox{\includegraphics[height=2cm, angle=0]{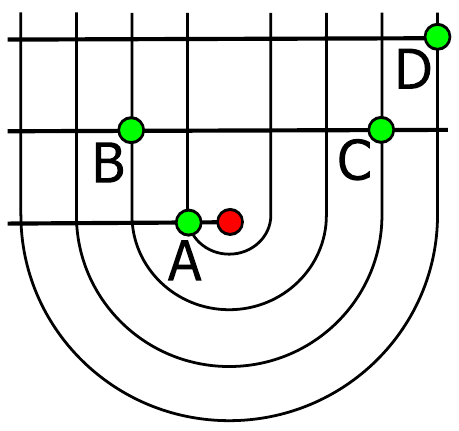}}}  \to \vcenter{\hbox{\includegraphics[height=2cm, angle=0]{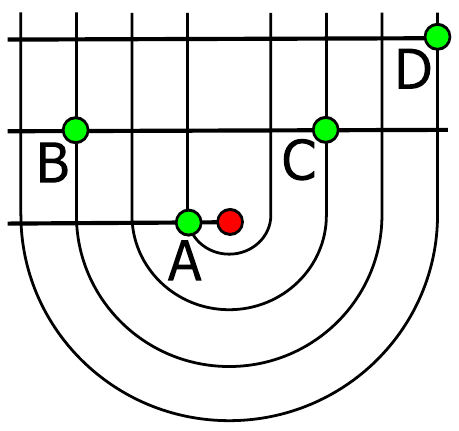}}} \]

We distinguish three cases, two of which have two sub-cases each, depending on which domain ($E$ or $F$) outside the local picture determines relative gradings of generators in the local picture:
\begin{enumerate}
\item \[
\vcenter{\hbox{\includegraphics[height=1.6cm, angle=0]{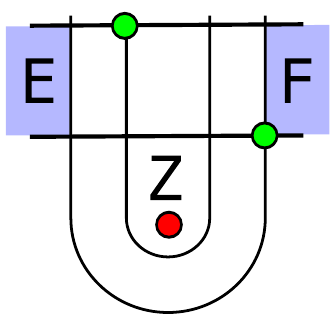}}} = \begin{cases}
q^{-2} E^{-1}\\ 
F \\
\end{cases} \vcenter{\hbox{\includegraphics[height=1.6cm, angle=0]{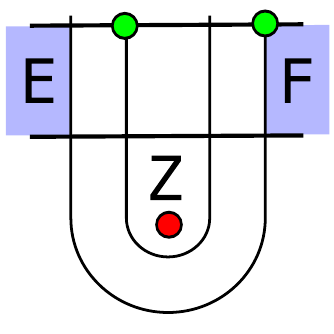}}} =\begin{cases}
Z E^{-1}\\ 
q^2 Z F \\
\end{cases} \vcenter{\hbox{\includegraphics[height=1.6cm, angle=0]{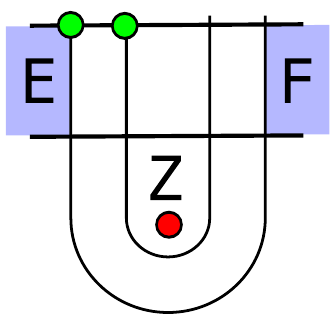}}} =\begin{cases}
E^{-1}\\ 
q^2 F \\
\end{cases} 
 \vcenter{\hbox{\includegraphics[height=1.6cm, angle=0]{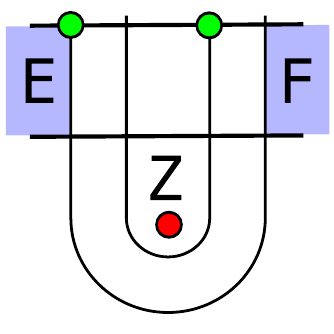}}} =q^2 \vcenter{\hbox{\includegraphics[height=1.6cm, angle=0]{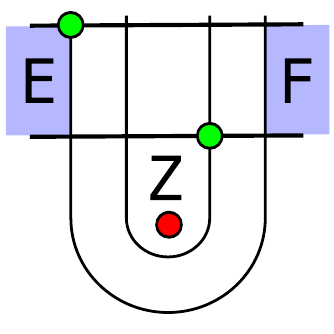}}}
\]
\item \[
\vcenter{\hbox{\includegraphics[height=1.6cm, angle=0]{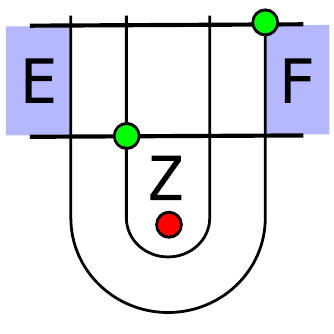}}} = \begin{cases}
q^{2} E\\ 
F^{-1}\\
\end{cases} \vcenter{\hbox{\includegraphics[height=1.6cm, angle=0]{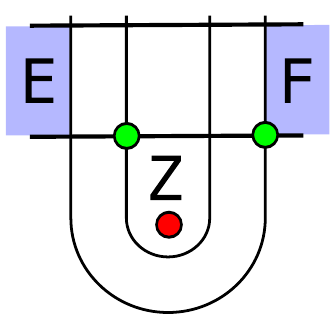}}} =\begin{cases}
q^4 Z E\\ 
q^2 Z F^{-1}\\
\end{cases} \vcenter{\hbox{\includegraphics[height=1.6cm, angle=0]{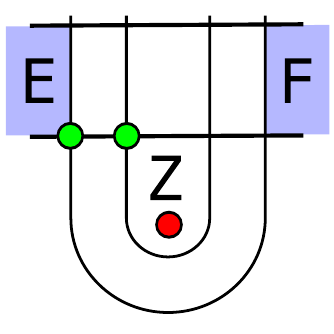}}} =\begin{cases}
q^4 E\\ 
q^2 F^{-1}\\
\end{cases} 
 \vcenter{\hbox{\includegraphics[height=1.6cm, angle=0]{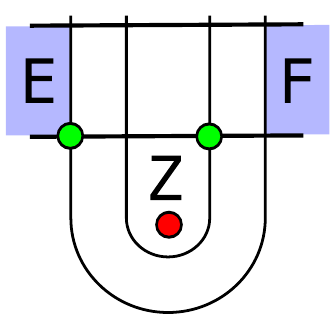}}} =q^2 \vcenter{\hbox{\includegraphics[height=1.6cm, angle=0]{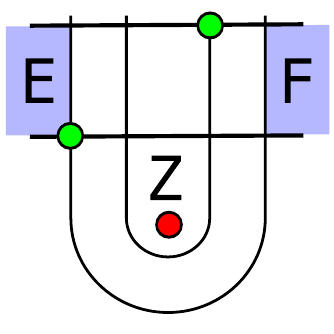}}}
\]
\item \[\vcenter{\hbox{\includegraphics[height=1.6cm, angle=0]{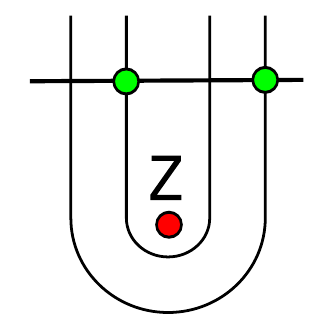}}} = q^2 Z \vcenter{\hbox{\includegraphics[height=1.6cm, angle=0]{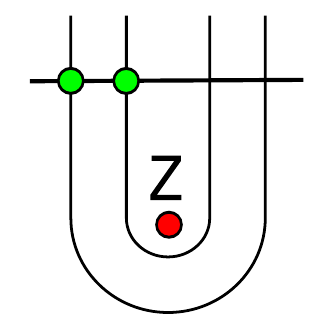}}} = q^2 \vcenter{\hbox{\includegraphics[height=1.6cm, angle=0]{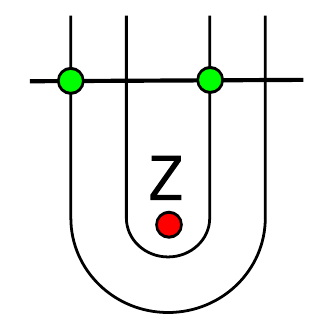}}}\]
\[\]
\end{enumerate} 
\end{proof}

\begin{lem} The simple slides in the geometric picture induce the grading shifts required by the twist rules.
\end{lem}
\begin{proof}
First note that it is sufficient to check this for simple slides between lowest $q$-grading configurations in the case of top twists and between highest $q$-grading configurations in the case of right twists. In writing the twist rules, we have renormalized the quantum binomial coefficient such that this relevant grading is $0$. Given this, it is straightforward to check that simple slides induce the same grading shifts as described by the twist rules.\\

For example, in 
\[TUP[i,j,k]\cong \sum_{h=k}^j t^{-h} s^{k} q^{k^2+h} {h \brack k}^+ UPs[i,j,h] \]
we see the dependence on $h$ is in a shift of $\frac{q}{t}$ when passing from a weight $h-1$ generator to the weight $h$ generator which is related by a simple slide, as expected for generators of type $UPs$.
\end{proof}

So far we have shown that the geometric algorithm accurately reproduces the relative gradings between the children of each parent. It remains to understand the relative gradings between children of different parents. 

\begin{defi} Each parent has a distinguished child that we call its \emph{clone}. If the parent is represented by the sequences $l,r$ of left and right primary intersections, then the clone is the child that is represented the reverse of $l$ concatenated with $r$. It is the child that arises by stealing zero verticals.  
\end{defi}

\begin{lem}
Clones of equal weight have the same relative gradings as their parents.
\end{lem}
\begin{proof}
Clones of equal weight have parents of equal weight. Their relative grading is computed by domains in $P_1$ which survive the twisting to $P_2$. The same domains thus compute the same relative gradings between clones.
\end{proof}

Combining the statement of the lemma with previous results we see that the geometric algorithm correctly computes the relative gradings between all children of parents of a certain weight. The last step in the proof of the main theorem is, therefore, to find the geometrically determined relative gradings of a set of children of parents of all possible weights and compare them with the twist rules. 

\begin{lem}

Let $p_0$  be an arbitrary weight $0$ parent and $c_0$ its clone. Write $p_k$ for the weight $k$ parent that is related to $p_0$ by simple slides and $c_k$ for its clone. Then the geometric algorithm correctly computes the relative gradings of the generators $c_0,\dots, c_j$.
\end{lem}
\begin{proof}
We choose temporary absolute gradings on the complexes associated to $P_1$ and $P_2$ and align them by requiring that passing from $p_0$ to its clone $c_0$ shifts grading by the amount described by the twist rules. We then check for each of the twelve cases (corresponding to the twelve rules) inductively on $k$ that the known shift from $p_{k-1}$ to $c_{k-1}$ together with rules (1) and (2) from the statement of Theorem \ref{mainthm} correctly compute the shift from $p_k$ to $c_k$.\\

We give two examples, one for a top twist and one for a right twist, and omit ten very similar cases.

\textbf{Case TUP:}
\[TUP[i,j,k]\cong \sum_{h=k}^j t^{-h} s^{k} q^{k^2+h} {h \brack k}^+ UPs[i,j,h] \]

In this case the parent is $p_k=UP[i,j,k]$ and the clone is the summand for $h=k$ on the right hand side of the twist rule above. In Figure \ref{figTUP} we have drawn the parents $p_{k-1}$ and $p_k$ in the first row, their clones $c_{k-1}$ and $c_k$ on the left and right in the second row and intermediate diagrams, computing the grading shift from $c_{k-1}$ to $c_k$, in the middle. Green dots on thick verticals represent tuples of intersections on tuples of verticals with size written above. The lower left shift comes from sliding the thin vertical to the left. The central lower shift comes from re-ordering the bent verticals. The lower right shift comes from the small disc containing $X^-$. The left vertical and the upper central grading shifts are known and the right vertical grading shift is determined by the commutativity of the diagram. It agrees with the shift described by the rule for $TUP$.

\begin{figure}[h]
\centerline{
\begindc{\commdiag}[20]
\obj(0,40)[1a]{$\vcenter{\hbox{\includegraphics[height=2cm, angle=0]{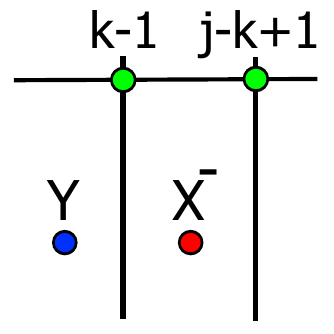}}}$}
\obj(135,40)[1b]{$\vcenter{\hbox{\includegraphics[height=2cm, angle=0]{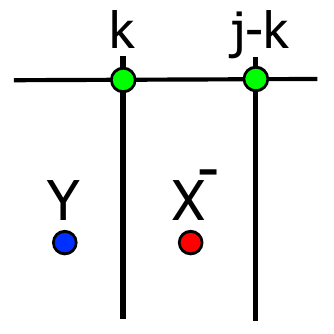}}}$}
\obj(0,0)[2a]{$\vcenter{\hbox{\includegraphics[height=2cm, angle=0]{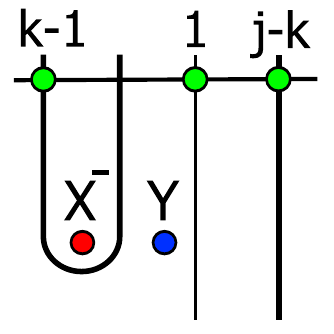}}}$}
\obj(45,0)[2b]{$\vcenter{\hbox{\includegraphics[height=2cm, angle=0]{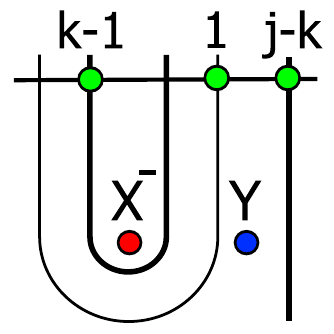}}}$}
\obj(90,0)[2c]{$\vcenter{\hbox{\includegraphics[height=2cm, angle=0]{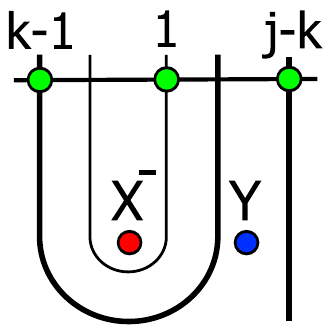}}}$}
\obj(135,00)[2d]{$\vcenter{\hbox{\includegraphics[height=2cm, angle=0]{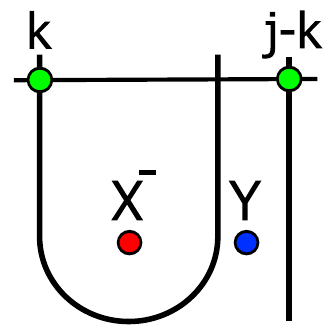}}}$}
\mor{1a}{1b}{$\frac{q s}{t}$}
\mor{1a}{2a}{$\frac{q^{k(k-1)} s^{k-1}}{t^{(k-1)}}$}
\mor{2a}{2b}{$\frac{q}{t}$}
\mor{2b}{2c}{$q^{2(k-1)}$}
\mor{2c}{2d}{$\frac{q^2 s^2}{t}$}
\mor{1b}{2d}{$\frac{q^{(k+1)k} s^k}{t^{k}}$}
\enddc
}
 \caption{}
 \label{figTUP}
\end{figure}

~\\
\textbf{Case ROP:}
For this situation see Figure \ref{figROP}. The known shift is on the right vertical arrow, the new shift is on the left vertical arrow and it is correctly computed by commutativity of the diagram.
 \begin{figure}[h]
\centerline{
\begindc{\commdiag}[20]
\obj(0,40)[1a]{$\vcenter{\hbox{\includegraphics[height=2cm, angle=0]{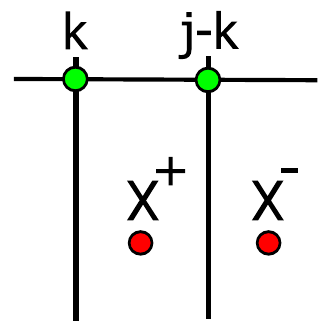}}}$}
\obj(135,40)[1b]{$\vcenter{\hbox{\includegraphics[height=2cm, angle=0]{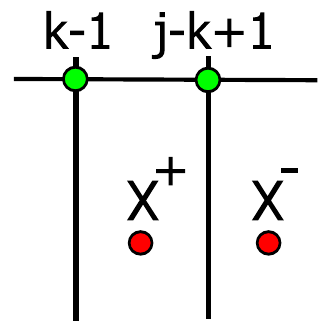}}}$}
\obj(0,0)[2a]{$\vcenter{\hbox{\includegraphics[height=2cm, angle=0]{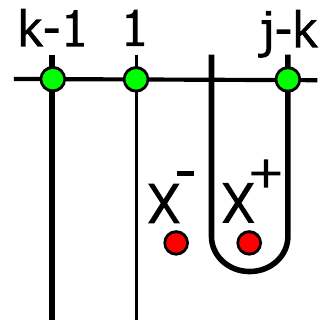}}}$}
\obj(45,0)[2b]{$\vcenter{\hbox{\includegraphics[height=2cm, angle=0]{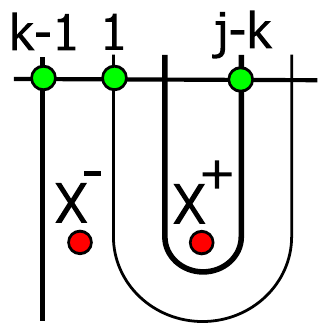}}}$}
\obj(90,0)[2c]{$\vcenter{\hbox{\includegraphics[height=2cm, angle=0]{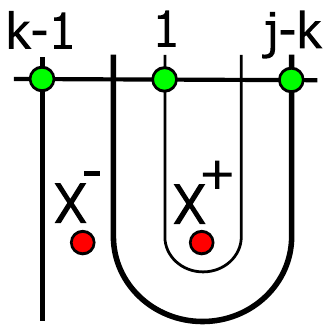}}}$}
\obj(135,0)[2d]{$\vcenter{\hbox{\includegraphics[height=2cm, angle=0]{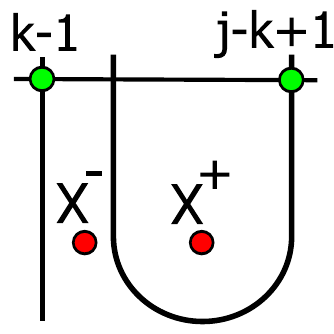}}}$}
\mor{1a}{1b}{$\frac{t q^{2j-1} s}{a}$}
\mor{1a}{2a}{$\frac{a^k q^{-k^2+k}}{ t^{k}}$}
\mor{2a}{2b}{$\frac{t}{q s}$}
\mor{2b}{2c}{$q^{2(k-j)}$}
\mor{2c}{2d}{$\frac{t q^{4j-2} s^2}{a^2}$}
\mor{1b}{2d}{$\frac{a^{k-1} q^{-(k-1)^2+k-1}}{t^{k-1}}$}
\enddc
}
 \caption{}
 \label{figROP}
\end{figure}
\end{proof}

\subsection{Differentials}
We have mentioned in section \ref{categ} 
that the differentials in the crossing complex are essentially uniquely determined as composition of an inclusion and the adjunction between $E$s and $F$s. The same holds for all complexes in the twist rules of section \ref{twistrules} by Lemma 4.3 of \cite{Cau2}. 
The colored HOMFLY complexes are simplifications of tensor products of crossing complexes. In our approach, these simplifications are computed iteratively by adding a crossing at a time via the twist rules. \\

It follows from the proof of Theorem \ref{mainthm} that the component of the differential coming from the last crossing added corresponds to sliding the innermost vertical from left to right, as described in the statement of the theorem. However, this has to be understood as differential that ties together the groups of generators that differ only by re-ordering verticals as in Lemma \ref{sortlem}. The actual differential between generators then can be computed as a composition of an inclusion, the differential on groups of generators and a projection.\\

The differentials coming from previous crossings can be identified with oppositely oriented simple discs in the picture, with boundary on $\alpha$ and verticals and containing one of the points $X^+$, $X^-$ and $Y$. As before, these differentials map between groups of generators that are related by re-ordering verticals. We have only managed to compute these differential explicitly in simple cases, but we expect that there are essentially only three types of morphisms involved, which depend on the special point  $X^+$, $X^-$ or $Y$ that is contained in the corresponding simple disc.

\subsection{Examples}
In this section we give two example computations. First, demonstrate the geometric algorithm we compute the chain groups in the colored HOMFLY complex of the rational tangle $T(3,1)$ with respect to colors $(\Lambda^i, \Lambda^2)$ with $i\geq 2$.\\
Second, we compute the Poincar\'e polynomial of $\mathfrak{sl}_N$ link homology of the $(\Lambda^i, \Lambda^j)$-colored Hopf link with $i>j$, reduced with respect to $\Lambda^i$.\\

\begin{exa}
The following figure shows the intersection points used for generators of weight $0$, $1$ and $2$ in the colored HOMFLY complex of the rational tangle $T(3,1)$.
\[\vcenter{\hbox{\includegraphics[height=3cm, angle=0]{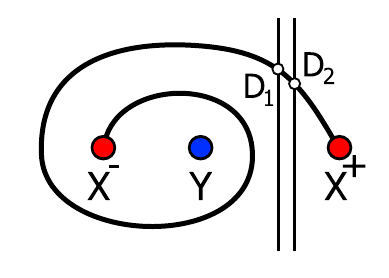}}}\quad \vcenter{\hbox{\includegraphics[height=3cm, angle=0]{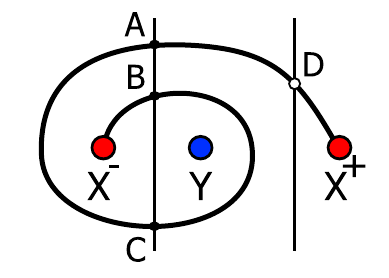}}} \quad \vcenter{\hbox{\includegraphics[height=3cm, angle=0]{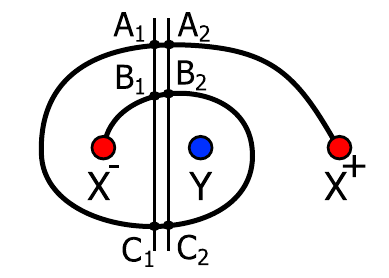}}}\]

The generators of the colored HOMFLY complex of this tangle are pairs of intersection points, one taken from each vertical. Here we have $13$ generators, which are shown in Figure \ref{figTref} together with the expected differentials between them. Vertical and horizontal arrows indicate differentials coming from the first and second crossing respectively, which correspond to the oppositely oriented simple discs containing the special points $Y$ and $X^-$. The diagonal arrows represent the differential coming from the last crossing; in the geometric picture this corresponds to sliding a left vertical to the right across $Y$. \\

 \begin{figure}[h]
\centerline{
\begindc{\commdiag}[15]
\obj(240,90)[DD]{$D_1D_2$}
\obj(160,50)[AD]{$AD$}
\obj(120,50)[CD]{$CD$}
\obj(120,90)[BD]{$BD$}
\obj(80,10)[AA]{$A_1A_2$}
\obj(40,10)[AC]{\begin{tabular}{ c }
  $A_1C_2$  \\
  $C_1A_2$
\end{tabular}}
\obj(0,10)[CC]{$C_1C_2$}
\obj(40,50)[AB]{\begin{tabular}{ c }
  $A_1B_2$  \\
  $B_1A_2$
\end{tabular}}
\obj(0,50)[BC]{\begin{tabular}{ c }
  $B_1C_2$  \\
  $C_1B_2$
\end{tabular}}
\obj(0,90)[BB]{$B_1B_2$}
\mor{AD}{DD}{$~$}[\atleft,0]
\mor{AA}{AD}{$~$}[\atleft,0]
\mor{AC}{CD}{$~$}[\atleft,0]
\mor{AB}{BD}{$~$}[\atleft,0]
\mor{CD}{AD}{$~$}[\atleft,0]
\mor{BD}{CD}{$~$}[\atleft,0]
\mor{AC}{AA}{$~$}[\atleft,0]
\mor{CC}{AC}{$~$}[\atleft,0]
\mor{BC}{AB}{$~$}[\atleft,0]
\mor{AB}{AC}{$~$}[\atleft,0]
\mor{BC}{CC}{$~$}[\atleft,0]
\mor{BB}{BC}{$~$}[\atleft,0]
\enddc
}
 \caption{}
 \label{figTref}
\end{figure}
In order to find the relative gradings of these generators we apply the rules of Theorem $\ref{mainthm}$. The non-additive part of the $q$-grading comes from re-ordering vertical, as shown in Lemma \ref{sortlem}. For example, there is a shift of $q^2$ from $C_1A_2$ to $A_1C_2$. In the figure we have written the lower $q$-grading configurations $C_1A_2$, $B_1A_2$ and $C_1B_2$ below the corresponding other configurations.\\

 The additive part of the grading can be read off from the type of arrows in the figure. Vertical, horizontal and diagonal arrows induces grading shifts of $\frac{t}{q^2}$, $\frac{t}{q^2s^2}$ and $\frac{t}{q}$ respectively. Hereby one has to be careful about identifying the generators between which the simple disc determines the grading shift without non-additive component. For example, in the lowest row in the figure the left arrow induces a shift of $\frac{t}{q^2s^2}$ between $C_1C_2$ and $A_1C_2$ while the right arrow induces a shift of equal magnitude between   $C_1A_2$ and $A_1A_2$.\\
 
 If we normalize the invariant by requiring $D_1D_2$ to lie in grading $a^0q^0s^0t^0$, then the other generators have gradings as shown in the table below. The non-additive part of the grading is written in bold font.

  \begin{center}
  \renewcommand{\arraystretch}{1.5}
    \begin{tabular}{ r | c | c | c | l  }
    
  ~ & $D$ & $A_2$ & $C_2$ & $B_2$ \\
  \hline
$A_1$ & $\frac{q}{t}$ & $\frac{q^2}{t^2}$ & $\frac{\mathbf{q^2} q^4s^2}{t^3}$ & $\frac{\mathbf{q^2} q^6s^2}{t^4}$ \\

$C_1$ & $\frac{ q^3s^2}{t^2}$ & $\frac{q^4s^2}{t^3}$ & $\frac{\mathbf{q^2} q^6s^4}{t^4}$ & $\frac{\mathbf{q^2} q^8s^4}{t^5}$ \\

$B_1$ & $\frac{ q^5s^2}{t^3}$ & $\frac{ q^6s^2}{t^4}$ & $\frac{\mathbf{q^4} q^8s^4}{t^5}$ & $\frac{\mathbf{q^4} q^{10}s^4}{t^6}$ \\
\end{tabular}
  \end{center}
 
\end{exa}

\begin{exa}  
As second example we compute the Poincar\'e polynomial of $\mathfrak{sl}_N$ link homology of the $(\Lambda^i, \Lambda^j)$-colored Hopf link with $i>j$, reduced with respect to $\Lambda^i$. We first consider the colored HOMFLY complex of the tangle $T(2,1)$ but suppress grading shifts:

\begin{figure}[ht]
\centerline{
\begindc{\commdiag}[10]
\obj(240,50)[O11]{$UP[i,j,0]$}
\obj(160,50)[O21]{$UP[i,j,1]$}
\obj(160,90)[O22]{$UP[i,j,1]$}
\obj(80,130)[O33]{$UP[i,j,2]$}
\obj(80,90)[O32]{${2 \brack 1}UP[i,j,2]$}
\obj(80,50)[O31]{$UP[i,j,2]$}
\obj(0,50)[O41]{$\cdots$}
\obj(0,90)[O42]{$\cdots$}
\obj(0,130)[O43]{$\cdots$}
\obj(320,50)[O10]{$TUPs[i,j,0]$}
\obj(320,90)[O20]{$TUPs[i,j,1]$}
\obj(320,130)[O30]{$TUPs[i,j,2]$}
\obj(320,170)[O40]{$~\cdots~$}
\mor{O20}{O10}{$~$}[\atleft,0]
\mor{O30}{O20}{$~$}[\atleft,0]
\mor{O40}{O30}{$~$}[\atleft,0]
\mor{O21}{O11}{$~$}[\atleft,0]
\mor{O31}{O21}{$~$}[\atleft,0]
\mor{O41}{O31}{$~$}[\atleft,0]
\mor{O32}{O22}{$~$}[\atleft,0]
\mor{O42}{O32}{$~$}[\atleft,0]
\mor{O43}{O33}{$~$}[\atleft,0]
\mor{O22}{O21}{$~$}[\atleft,0]
\mor{O33}{O32}{$~$}[\atleft,0]
\mor{O32}{O31}{$~$}[\atleft,0]
\enddc
}
 \caption{}
 \label{figHopf}
\end{figure}
The $k^{th}$ row from the bottom in the left diagram in Figure \ref{figHopf} is the complex $TUPs[i,j,k]$. One way to compute the $\Lambda^i$-reduced $\mathfrak{sl}_N$ link homology of the colored Hopf link would be to close off the $j$-colored strand in all webs in the colored HOMFLY complex on the left side of Figure \ref{figHopf}, replace each of these webs with the isomorphic direct sum of copies of the $i$-colored strand $UP[i,0,0]$, compute the induced differentials and apply Gaussian elimination to cancel all acyclic summands. However, it is much easier to do this one row at a time. For this we close off the $j$-colored strand in the knotted webs $TUPs[i,j,k]$ in the complex shown in the right column in Figure \ref{figHopf} to get a complex of knotted webs $ClTUPs[i,j,k]$. For each of the knotted webs in this chain complex, the following isomorphisms hold up to grading shifts:\footnote{Here we assume that the categorical tangle invariants extend to invariants of knotted webs. This is known for the matrix factorization construction \cite{Wu}, and is expected to extend to the general setting via the uniqueness results in \cite{Cau2}.}
\comm{
\begin{equation*}
ClTUPs[i,j,k] = \vcenter{\hbox{\includegraphics[height=2.5cm, angle=0]{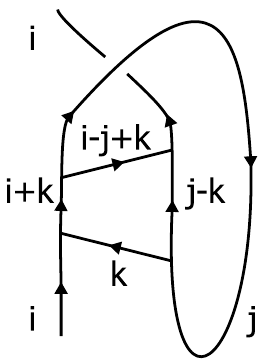}}} \cong \vcenter{\hbox{\includegraphics[height=2.5cm, angle=0]{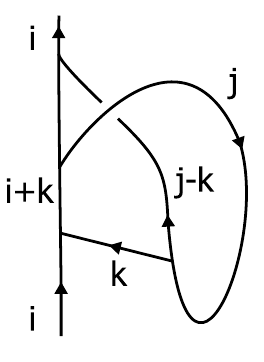}}} \cong\vcenter{\hbox{\includegraphics[height=2.5cm, angle=0]{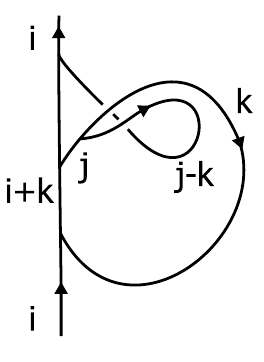}}} \cong\vcenter{\hbox{\includegraphics[height=2.5cm, angle=0]{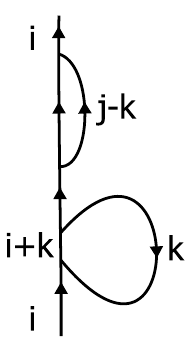}}} \cong {N-i\brack k}{i \brack j-k}\vcenter{\hbox{\includegraphics[height=2.5cm, angle=0]{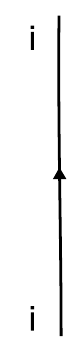}}}   \end{equation*}
}
\begin{align*}
ClTUPs[i,j,k] &= \vcenter{\hbox{\includegraphics[height=2.5cm, angle=0]{hopfcomp1.pdf}}} \cong \vcenter{\hbox{\includegraphics[height=2.5cm, angle=0]{hopfcomp2.pdf}}} \cong\vcenter{\hbox{\includegraphics[height=2.5cm, angle=0]{hopfcomp3.pdf}}}\\
&\cong \vcenter{\hbox{\includegraphics[height=2.5cm, angle=0]{hopfcomp4.pdf}}} \cong {N-i\brack k}{i \brack j-k}\vcenter{\hbox{\includegraphics[height=2.5cm, angle=0]{hopfcomp5.pdf}}} = {N-i\brack k}{i \brack j-k} UP[i,0,0]    
\end{align*}

Thus, each of the complexes $ClTUPs[i,j,k]$ is homotopy equivalent to a complex concentrated in one homological degree only. We next replace each row in the (closed off) colored HOMFLY complex by the corresponding 1-term complex computed above. So far we have disregarded grading shifts, but from the decategorified invariant, it is easy to work out that the remaining terms are concentrated in the highest homological grading in each row, see Figure \ref{figHopf2}, and that all differentials thus must be trivial.

\comm{
\obj(240,0)[O11]{${N-i\brack 0}{i \brack j-0}\vcenter{\hbox{\includegraphics[height=2cm, angle=0]{hopfcomp5.pdf}}}  $}
\obj(160,0)[O21]{$0$}
\obj(160,60)[O22]{${N-i\brack 1}{i \brack j-1}\vcenter{\hbox{\includegraphics[height=2cm, angle=0]{hopfcomp5.pdf}}}  $}
\obj(80,120)[O33]{${N-i\brack 2}{i \brack j-2}\vcenter{\hbox{\includegraphics[height=2cm, angle=0]{hopfcomp5.pdf}}}  $}
}
\begin{figure}[h]
\centerline{
\begindc{\commdiag}[10]
\obj(240,0)[O11]{${N-i\brack 0}{i \brack j-0} UP[i,0,0]  $}
\obj(160,0)[O21]{$0$}
\obj(160,60)[O22]{${N-i\brack 1}{i \brack j-1}UP[i,0,0] $}
\obj(80,120)[O33]{${N-i\brack 2}{i \brack j-2}UP[i,0,0] $}
\obj(80,60)[O32]{$0$}
\obj(80,0)[O31]{$0$}
\obj(0,0)[O41]{$\cdots$}
\obj(0,60)[O42]{$\cdots$}
\obj(0,120)[O43]{$\cdots$}
\mor{O21}{O11}{$~$}[\atleft,0]
\mor{O31}{O21}{$~$}[\atleft,0]
\mor{O41}{O31}{$~$}[\atleft,0]
\mor{O32}{O22}{$~$}[\atleft,0]
\mor{O42}{O32}{$~$}[\atleft,0]
\mor{O43}{O33}{$~$}[\atleft,0]
\mor{O22}{O21}{$~$}[\atleft,0]
\mor{O33}{O32}{$~$}[\atleft,0]
\mor{O32}{O31}{$~$}[\atleft,0]
\enddc
}
 \caption{}
\label{figHopf2}
\end{figure}

 Tracking the $q$-grading through the above computation or comparing with the decategorified invariant shows:
\begin{prop} The Poincar\'e polynomial of $\mathfrak{sl}_N$ link homology of the $(\Lambda^i, \Lambda^j)$-colored Hopf link with $i>j$, reduced with respect to $\Lambda^i$ is, up to multiplication by a monomial:
\[\mathcal{P}^N_{i,j}(Hopf,\Lambda^i)= \sum_{k=0}^j t^{2k} q^{k(2+N)}  {N-i\brack k}{i \brack j-k}.\]
\end{prop}
\end{exa}

\subsection{Comparison with Bigelow's and Manolescu's picture}
\label{bigelow}
In this section we explain similarities between the geometric picture described in section \ref{geom} and Bigelow's geometric model for the Jones polynomial \cite{Big1} and Manolescu's extension \cite{Man1} to a model for the generators of a chain complex computing Seidel-Smith homology (symplectic Khovanov homology) \cite{SeS}. 

Before going into details about the similarities, we want to mention the most visible differences between the construction in this paper and Bigelow's and Manolescu's construction. First of all, they work with arbitrary knots and links that are presented as closures of braids, while our geometric algorithm is (so far) restricted to rational tangles. Their picture computes generators for uncolored (i.e. $\Lambda^1$-colored) $\mathfrak{sl}_{2}$ chain complexes and it says very little about differentials. Our picture, on the other hand, computes invariants for arbitrary fundamental $\mathfrak{sl}_N$ representations and at least some components of the differential can be read off. Finally, their theory is a reduced one, yielding a one-dimensional invariant for the unknot, while ours is closer to an unreduced theory, see Section \ref{1stproof}. Note that Bigelow \cite{Big2} and Manolescu \cite{Man2} also have corresponding theories for $\mathfrak{sl}_N$, but whether they are related to each other as in the $\mathfrak{sl}_{2}$ case or to the construction here is unclear. \\

We now rephrase our geometric algorithm in the language of  \cite{Big1}. Let $M := D^2 \setminus \{X^+, X^-, Y\}$ be the usual disc in $\mathbb{C}$ with the three special points $X^+$, $X^-$ and $Y$ removed. The weight $h$ part of the colored HOMFLY complex of a colored rational tangle can be interpreted as graded intersection number of the submanifolds $A:= Sym^j(\alpha)\setminus \Delta$ and $V_h:= w_1\times \cdots w_h \times u_1\times \cdots \times u_{j-h}$ in $Conf^j(M):= Sym^j(M)\setminus \Delta$. Here $\Delta$ denotes the appropriate big diagonal. Similarly as in Bigelow's setting, the graded intersection number can be described as algebraic intersection number of lifts of $A$ and $V_h$ to a covering space specified by a surjective homomorphism $\Phi\colon \pi_1(Conf^j(M))\to \mathbb{Z}^4$. Relative gradings of intersection points in this picture can be computed by taking a loop $\gamma$ in $Conf^j(M)$ that starts at one intersection point, travels to the second intersection point along $A$ and returns along $V_h$, then the grading difference is $\Phi([\gamma])$. 
Such a loop $\gamma$ connecting intersection points $(x_1,\dots, x_j)$ and $(y_1,\dots, y_j)$ of $A$ and $V_h$ can be represented by a $j$-tuple of paths $\gamma_k$ starting at $x_k$, proceeding to $y_{\sigma(k)}$ along $\alpha$ and further to $x_{\sigma(k)}$ along the $\sigma(k)$th vertical, where $\sigma \in S_j$ is a permutation. In this picture, $\Phi$ computes a linear combination of winding numbers of the paths $\gamma_k$ around the points $X^+$, $X^-$, $Y$ and around each other.\\

$\Phi$ can be constructed in a similar way as in the papers of Bigelow \cite{Big1} and Manolescu \cite{Man1} to reproduce exactly the behaviour described in Theorem \ref{mainthm}. As an example we explain how to count the winding of arcs around $X^+$ and around each other:

\begin{exa} Define the two homomorphisms:
 \[\Phi_1\colon \pi_1(Conf^j(M))\to \pi_1(Conf^j(D^2)) = Br_{j}\to \mathbb{Z}\] 
 \[\Phi_2\colon \pi_1(Conf^j(M))\to \pi_1(Conf^j(D^2\setminus X^+))\to\pi_1(Conf^{j+1}(D^2)) = Br_{j+1}\to \mathbb{Z}\]
The first maps in both lines are induced by inclusion. The second map in the second line comes from adding the point $X^+$ to the unordered $j$-tuple. The last maps in both lines are the natural abelianization maps. Then $(\Phi_2-\Phi_1)/2 \colon \pi_1(Conf^j(M))\to \mathbb{Z}$ is a homomorphism and it counts the winding of the arcs $\gamma_k$ around $X^+$, while $\Phi_1$ alone counts twice the winding of arcs around each other. 
 \end{exa}
 
An alternative description is that $(\Phi_2-\Phi_1)/2$ and $\Phi_1$ count the winding of $\gamma$ around the divisor $X^+\times Sym^{j-1}(D^2) \subset Sym^j(D^2)$ and around the big diagonal $\Delta\subset Sym^j(D^2)$ respectively.\\

Once one knows how to count winding around the divisors of the special points and the diagonal, it is easy to assemble the correct $\Phi\colon \pi_1(Conf^j(M))\to \mathbb{Z}^4$. For example, the contribution coming from winding around the diagonal, which is exactly the non-additive part $q-Q$ of the $q$-grading, contributes $2 \Phi_1$ to the $q$-coordinate of $\Phi$. 

\begin{exa} 
The main step in the proof of Lemma \ref{sortlem} describes a half twist of arcs around each other, which causes a shift of $q^2$.
\end{exa}

\begin{exa} A oppositely oriented simple disc intersecting just one other primary intersection in its interior corresponds to a full twist of arcs around each other and causes a non-additive $q$-grading shift of $q^4$, in agreement with the statement of Theorem \ref{mainthm}.
\end{exa}

\begin{rem}
We require the path $\gamma$ that is used to compute the winding number around certain divisors to lie in $Conf^j(M)$. In particular, the part of $\gamma$ that lies on $Sym^j(\alpha)$ is required to be disjoint from the diagonal. In the statement of Theorem \ref{mainthm} and the interpretation of differentials in the colored HOMFLY complex, on the other hand, we do use paths $\gamma$ that can have intersections with the diagonal. 

The problematic cases are exactly the ones where a simple disc $D_s$ representing a differential between two generators, has an intersection with another primary intersection along its boundary $\gamma = \partial D_s$. Theorem \ref{mainthm} tells us there should be a contribution of $q^2$ for each such intersection. 

Before we can verify that this contribution comes from winding around the diagonal $\Delta$, we have to push $\gamma$ off $\Delta$. Here (and in the generic case) $\gamma$ only hits the multiplicity $2$ part of the diagonal and, hence, we can restrict to the model case of $Sym^2(\alpha) \subset Sym^2(D^2)$ where we can explicitly  describe the canonical push-off $\gamma'$ of $\gamma$. Figure \ref{figdisc} shows the typical situation of a simple disc $D_s$ (shaded) intersecting another primary intersection (blue dot) along its boundary. 
\begin{figure}[h]
\centerline{
\begindc{\commdiag}[10]
\obj(0,0)[A]{$\vcenter{\hbox{\includegraphics[height=2cm, angle=0]{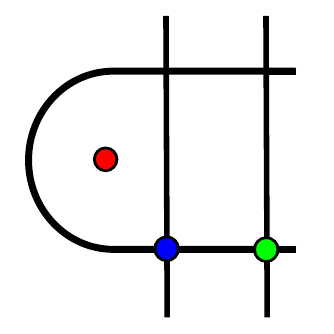}}}$}
\obj(160,0)[B]{$\vcenter{\hbox{\includegraphics[height=2cm, angle=0]{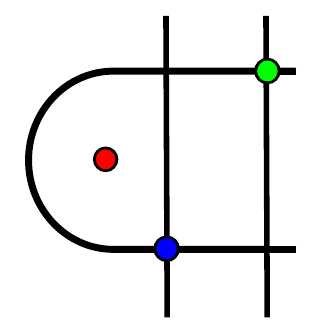}}}$}
\obj(80,30)[C]{$\vcenter{\hbox{\includegraphics[height=1.6cm, angle=0]{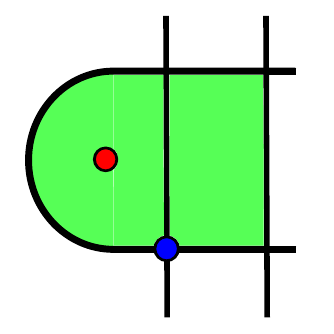}}}$}
\mor{A}{B}{$~$}[\atleft,0]
\enddc
}
\caption{}
\label{figdisc}
\end{figure}

Theorem \ref{mainthm} tells us that the grading difference of the generators shown on both sides of the figure consists of a contribution from winding around the special point (red dot) and a contribution of $q^2$ from the intersection of the simple disc $D_s$ with the other primary intersection (blue dot). The standard loop $\gamma = \partial D_s$ is given by moving the green dot from its position in the left image around the left bend (producing the right image) and back along its vertical while keeping the blue dot fixed. $\gamma$ intersects the diagonal when the two dots coincide.

The canonical push-off $\gamma'$, which is disjoint from the diagonal, is given by the following two arcs. The green dot moves left from its position in the left image to the original position of the blue dot while the blue dot itself moves left and around the bend (which produces the right image with colors swapped) and then down along the vertical (which produces the left image again, but with colors swapped). 

By forgetting the special point (red dot), it is easy to see that the arcs representing $\gamma'$, viewed as a braid in $Br_2$ are just the braid group generator. Thus $\Phi_1(\gamma')=1$ as required for the contribution $q^2$. Also, $\gamma'$ still winds once around the divisor of the special point.
\end{rem}

\section{The color stable HOMFLY polynomial}
\label{conj2}
This section gives two proofs for Conjecture \ref{conjA} on the decategorified level of polynomial HOMFLY invariants. The first one uses skew Howe duality and tries to stay as close to the categorified setting as possible. The second proof uses skein theory and proves Proposition \ref{colshiftprop} for arbitrary links with an unknot component. We define the color stable HOMFLY polynomial of a link with an unknot component and prove that for 2-component links it specializes to the multivariable Alexander polynomial. Finally we compare the color stable HOMFLY polynomial with multivariable link invariants arising from the Lie superalgebras $\mathfrak{sl}_{m|n}$ as described in \cite{GP1}, \cite{GP2} and \cite{GPT}.

\subsection{First proof}
\label{1stproof}
We start by giving a proof for the decategorified Conjecture 1 for rational links that stays as close as possible to the categorified version.\\

Suppose we are given an $(\Lambda^i,\Lambda^j)$-colored rational two-component link that can be written as closure of a positive rational tangle with all-upwards boundary orientations. Then the Poincar\'e polynomial of this tangle decategorifies by setting $t=-1$ to a $\mathbb{Z}[a^{\pm 1},q^{\pm 1},s^{\pm 1}]$-linear combination of webs $UP[i,j,k]$.\\

The next step is to take the closures of the webs $UP[i,j,k]$ and evaluate them to elements of the ground ring, for example by using relations \eqref{rel2} and \eqref{rel3} in the $\mathfrak{sl}_N$ spider category. We use this opportunity to demonstrate, as an alternative, the simplification process described in section \ref{simpl}.\\

The closure of $UP[i,j,k]$ is the skew Howe image of  $E_3^{(j)}F_1^{(i)}E_2^{(k)}F_2^{(k)}E_1^{(i)}F_3^{(j)}1_{\lambda}\in _{\mathcal{A}}\dot{U}(\mathfrak{sl}_{4})$ where $\lambda$ corresponds to the sequence $(N,0,0,N)$. Using the commutation relations for $E$s and $F$s and the fact that some weight spaces for the $\mathfrak{sl}_{4}$ action are trivial, we can simplify this expression:\footnote{For clarity we use as subscript in $1_{(a,b,c,d)}$ the sequence $(a,b,c,d)$ instead of the corresponding weight $\mu$.}
\begin{align*} &E_3^{(j)}F_1^{(i)}E_2^{(k)}F_2^{(k)}E_1^{(i)}F_3^{(j)}1_{(N,0,0,N)}
\cong E_3^{(j)}E_2^{(k)}F_1^{(i)}E_1^{(i)}1_{(N,k,j-k,N-j)}F_2^{(k)}F_3^{(j)} \\
&\cong {N-k \brack i} E_3^{(j)}E_2^{(k)}F_2^{(k)}1_{(N,0,j,N-j)}F_3^{(j)} \cong {N-k \brack i}{j \brack k} E_3^{(j)} F_3^{(j)}1_{(N,0,0,N)}\\ 
&\cong {N-k \brack i}{j \brack k}{N \brack j} 1_{(N,0,0,N)} \cong {N-k \brack i}{j \brack k}{N \brack j} 1_{(0,0,N,N)}    
\end{align*}

In the HOMFLY evaluation we replace ${N+b \brack c}$ by ${b \brack c}_a:= \prod_{k=1}^c \frac{a q^{b-k+1}-a^{-1}q^{-b+k-1}}{q^{k}-q^{-k}}$. The closure of $UP[i,j,k]$ thus evaluates to ${-k \brack i}_a{j \brack k}{0 \brack j}_a = {-k \brack j-k}_a{-i \brack k}_a{0 \brack i}_a$.\\

In order to get the behaviour claimed in Conjecture 1, we have to reduce with respect to the higher color $\Lambda^i$. On the decategorified level, this just means dividing by ${0\brack i}_a$, the invariant of the $\Lambda^i$-colored unknot. We can write the reduced evaluation of the closure of $UP[i,j,k]$ as:
\[{-k \brack j-k}_a{-i \brack k}_a = {-k \brack j-k}_a \prod_{l=1}^k \frac{a q^{-i-l+1}-a^{-1}q^{i+l-1}}{q^{l}-q^{-l}}= {-k \brack j-k}_a \prod_{l=1}^k \frac{a s^{-1} q^{-j-l+1}-a^{-1}s q^{j+l-1}}{q^{l}-q^{-l}}  \]

This shows that in the reduced case, not only the coefficients coming from the grading shifts in the colored HOMFLY complex, but also the evaluation of $UP[i,j,k]$  depends in a controlled way on the higher color $i$. More precisely, we have shown that there exists a three-variable invariant\footnote{Up to multiplication by a monomial.} of two-component rational links, which takes values in $\mathbb{Z}[a^{\pm 1},s^{\pm 1}](q)$, and for any $i\geq j$ it specializes to the higher-color reduction of the $(\Lambda^i,\Lambda^j)$-colored HOMFLY polynomial under setting $s=q^{i-j}$.

\begin{rem}
The method of computing colored HOMFLY polynomial of 2-bridge links explained in this subsection in fact provides explicit q-holonomic formulas in the sense of \cite{Gar}. For more details and a Wolfram Mathematica implementation of this algorithm see \cite{Wed}.
\end{rem}

\subsection{Second proof}
\label{2ndproof}
A different proof of Conjecture 1 on the decategorified level is possible via skein theory. In fact, this proof works for arbitrary colored links $L'$ with a $\Lambda^i$-colored unknot component $U$. Let $L=L'\setminus U$ which we consider to live in a solid torus.

The idea is to compute the colored HOMFLY polynomial of $L'$ in two steps. First one evaluates $L$ in the skein of the annulus onto which the solid torus projects. Second one pairs this again with the $\Lambda^i$ colored unknot. 

\begin{defi} Let $F$ be an open subset of $\mathbb{R}^2$. Then define $S(F)$ to be the free $\mathbb{Z}[a^{\pm 1}](q)$-module spanned by closed webs embedded in $F$, modulo web relations inside $F$. $S(F)$ is called the HOMFLY skein of $F$.

If $G\subset F$ is an inclusion of open subsets of $\mathbb{R}^2$, then there is a canonical homomorphism $S(G)\to S(F)$ by interpreting webs as lying in $F$. The homomorphism $S(F)\to S(\mathbb{R}^2)$ is called evaluation and is denoted by $\langle . \rangle$.
\end{defi}

\begin{exas}~
\begin{enumerate}
\item $S(\mathbb{R}^2)$ is free of rank 1 over $\mathbb{Z}[a^{\pm 1}](q)$ and is spanned by the empty diagram.
\item The skein of the annulus $S(A)$ has the structure of a commutative algebra, where multiplication is given by stacking two annuli inside each other. 
\end{enumerate}
\end{exas}

By projecting a link $L$ lying in $F\times I$ onto $F$ and replacing crossings via the formulas in section \ref{RT}, $L$ can be regarded as an element of the skein $S(F)$. Up to multiplication by a scalar depending on framing, this is well defined. 

\begin{lem} 
As an algebra the skein of the annulus is freely generated by the set $\{d_j \mid j\in \mathbb{Z}\}$ where $d_j$ is given by a $S^{|j|}$-colored longitudinal unknot in counter clockwise (clockwise) orientation if $j>0$ ($j<0$) and $d_0$ is the empty diagram. Here $S^k$ stands for the one-row Young diagram with $k$ boxes; for $\mathfrak{sl}_N$ this corresponds to the $k^{th}$ symmetric power of the standard representation.

Another free generating set is given by $\{\phi_j \mid j\in \mathbb{Z}\}$ where $\phi_j$ is the closure of the $|j|$-strand braid $\sigma_{|j|-1}\cdots \sigma_{1}$, written in standard braid group generators, with orientation counter clockwise (clockwise) if $j>0$ ($j<0$). 
\end{lem}
\begin{proof} The second generating set is due to Turaev \cite{Tu}. This also shows that $S(A)$ splits as an algebra into a product $S(A)=S(A)^+ \times S(A)^-$ of isomorphic algebras which are freely generated by $\{\phi_j \mid j\in \mathbb{N}\}$ and $\{\phi_{-j} \mid j\in \mathbb{N}\}$ respectively. Lukac \cite{Luk} showed that $S(A)^{\pm}$ are isomorphic to the ring of symmetric functions on a countably infinite alphabet with the $i^{th}$ complete (elementary) symmetric function corresponding to a $S^i$- ($\Lambda^i$-) colored unknot $d_i$ ($c_i$) with the appropriate orientation. See also \cite{Ais}.
\end{proof}

Given an element $X$ of $S(A)$ and $i\in \mathbb{Z}$ we can get a new element $\psi_i(X)$ of $S(A)$ by linking $X$ with a meridional $\Lambda^i$-colored unknot $c_i$.

\begin{lem} 
For $i\in \mathbb{Z}$ there exist algebra homomorphisms $ t_i \colon S(A)\to \mathbb{Z}[a^{\pm 1}](q)$ given by 
\[ t_i(X)= \frac{\langle \psi_i(X)\rangle}{\langle c_i \rangle}.
\]
\end{lem}
\begin{proof}
See \cite{MoL} sections 1.4 and 1.5.
\end{proof}

Recall that we have decomposed $L'$ into a $\Lambda^i$-colored unknot $U$ and some remainder $L$ in a solid torus $A\times I$. By taking a projection of $L$ onto the annulus and applying the crossing replacement rules, $L$ evaluates to some element of $S(A)$, which we denote by $\pi(L)$. The colored HOMFLY polynomial of $L'$ is then $\langle \psi_i(\pi(L))\rangle$ and its reduction with respect to color $\Lambda^i$ is $t_i(\pi(L))$. In order to prove Conjecture 1 it suffices to show that $S(A)$ has a generating set that behaves well under color shift.

\begin{prop} There exist functions $p_j\in \mathbb{Z}[a^{\pm 1},s^{\pm 1}](q)$ such that for the generating set $\{d_j\mid j\in \mathbb{Z}\}$ we have:
\[t_i(d_j) = p_j(a, s = q^i,q) \quad \forall i\geq 0.\]
\end{prop}
\begin{proof}
This follows readily from \cite{MoL} Lemma 3.1 where for $i,j>0$ the authors prove the first equality in the following computation (with different notation): 
\[t_i(d_j) = \langle d_j \rangle \frac{a-a^{-1}(q^{j} - q^{j-i} + q^{-i})}{a-a^{-1}} = \langle d_j \rangle \frac{a-a^{-1}(q^{j} - q^{j}s^{-1} + s^{-1})}{a-a^{-1}} .\] 
The right hand side is really polynomial in $a$ because
\[\langle d_j \rangle = \prod_{k=0}^{|j|-1} \frac{a q^{-k}-a^{-1} q^{k}}{q^{k+1} - q^{-k-1}},\] 
which is also proved in \cite{MoL}. The cases for other signs of $i$ or $j$ are similar.
\end{proof}

\begin{rem} Alternatively one can deduce the statement of the proposition for the generating set $\{\phi_j\mid j\in \mathbb{Z}\}$ from the already established decategorified Conjecture 1 for rational links. For this note that $t_i(\phi_j)$ is the $\Lambda^i$-reduced colored HOMFLY polynomial of the $(\Lambda^i,\Lambda^1)$-colored $(2,2j)$ torus link. 
\end{rem}

\subsection{The color stable HOMFLY polynomial and the multivariable Alexander polynomial.}

Let $L' = L \cup U$ be a link with an unknot component $U$ and some arbitrary coloring on $L$.

\begin{defi} 
The \emph{color stable HOMFLY polynomial} $P^{st}(L',U)\in \mathbb{Z}[a^{\pm 1},s^{\pm 1}](q)$ of $L'$ with respect to the unknot component $U$ is the unique element of $\mathbb{Z}[a^{\pm 1},s^{\pm 1}](q)$ satisfying:
 \[P^{st}(L',U)(a,s=q^i,q) = t_i(L)\quad \forall i\geq 0 .\]
\end{defi}

\begin{thm} \label{MVAthm} Let $L' = K \cup U$ be a two-component link with an unknot component $U$ and suppose $K$ is colored by $\Lambda^1$. Then:
\[P^{st}(L',U)(1, u/v, v) = (1-u^2)\Delta(L')(u^2,v^2)\] where $\Delta(L')(x,y)$ is the multivariable Alexander polynomial of $L'$ with $U$ labelled by $x$.   

\end{thm}

We first prove this for rational links and show how the computation of the multivariable Alexander polynomial ties in with the geometric algorithm.

\begin{lem} The theorem is true for rational links $L'$.
\label{ratMVA}
\end{lem}
\begin{proof} 
Actually we prove: \[(q-q^{-1})P^{st}(L',U)\mid_{a=1,~s=u/v,~ q=v} = (1-v^2)(1-u^2)\Delta(L')(u^2,v^2)\]

We may assume that $L'$ is the closure of a rational tangle with odd length of continued fraction expansion. Since we work with a $(\Lambda^j,\Lambda^1)$-colored tangle, we only need one vertical, left or right, in the geometric picture. We assume that the tangle has all upward boundary orientation and the colored HOMFLY complex has objects $UP[j,1,1]$ and $UP[j,1,0]$. The case of the other orientation with objects $OP[j,1,1]$ and $OP[j,1,0]$ is analogous.\\

In section \ref{2ndproof} we have computed the reduced closures of the objects $UP[j,1,k]$. Their contributions to the left hand side of the above equation are computed as follows:
\[\text{Reduced closure of} 
\begin{cases}
UP[j,1,0] \\
UP[j,1,1]
\end{cases} = 
\begin{cases}
{0\brack 1}_a{-j\brack 0}\\
{-1\brack 0}_a{-j\brack 1}
\end{cases}\to\begin{cases}
a-a^{-1} \\
a q^{-1}s^{-1}-a^{-1}q s
\end{cases}\to \begin{cases}
0\\
u^{-1}(1-u^2)
\end{cases}\]
Here the first arrow is multiplication by $q-q^{-1}$ and the second is the substitution $a=1, s= u/v, q=v$.\\

In the geometric algorithm we, hence, only need to count intersection points with the left vertical. Note that because the continued fraction expansion has odd length, the last twist applied to the diagram was a top twist. Thus the intersection points with the left vertical are paired up by simple discs and their relative grading is $-v^2$. We can replace two paired intersection points by the single intersection point of $\alpha$ with the real axis along the segment of $\alpha$ that joins the pair. The contribution of such a double intersection point is thus $u^{-1}(1-v^2)(1-u^2)$. The relative gradings of such double intersection points can be computed from winding numbers of connecting paths around the special points $Y,X^-,X^+$, which count as $-v^2$, $-u^2$ and $-u^{-2}$ under the substitution. Here a connecting path starts at one double intersection point on the real axis, travels along $\alpha$ to the second double intersection point and returns on the real axis. This shows that, up to multiplication by a monomial, the geometric algorithm for the the modified colored HOMFLY polynomial outputs $(1-u^2)(1-v^2)P(u^2,v^2)$ where $P$ is some two-variable polynomial. It remains to show that $P=\Delta(L')$.\\

A classical way of computing the multivariable Alexander polynomial is via Fox calculus on a presentation of the fundamental group of the link complement. Since our diagrams are essentially genus two Heegaard diagrams of the link complement, we can extract a presentation $\langle u^2, v^2 \mid w =1\rangle$ for its fundamental group by the following procedure. First we have to replace the arc $\alpha$ by the embedded circle $\overline{\alpha}$ which is the boundary of a small neighbourhood of $\alpha$. Starting from any point on $\overline{\alpha}$ the word $w$ is assembled from letters $\{u^2,u^{-2},v^2, v^{-2}\}$ by appending a letter $v^{\pm 2}$ for every intersection with the segment $[-2, -1]$ on the real axis, where the exponent depends on whether $\alpha$ hits the real axis from above or below, and similarly $u^{\pm 2}$ for intersections with the segment $[1,2]$. 
Then the multivariable Alexander polynomial can be extracted from the presentation, by taking the Fox derivative of $w$ with respect to the variable $v^2$ and then dividing by $(1-u^2)$.\\

\textbf{Claim:} The summands produced by the Fox derivative are in bijection with the intersection points of $\overline{\alpha}$ with the segment $[-2,-1]$ on the real axis and their relative gradings are determined by the winding of connecting paths around the special points $Y,X^-,X^+$, which count as $-v^2$, $-u^2$ and $-u^{-2}$. The connecting paths run along $\overline{\alpha}$ from one intersection point to the other and back on the real axis. 
The proof of this claim is an exercise for the reader who is familiar with the Fox derivative.\\

One can further simplify this picture by noting that in our case intersection points of $\overline{\alpha}$ with $[-2,-1]$ always come in pairs that correspond to an intersection of $\alpha$ with $[-2,-1]$. Furthermore, these pairs have relative grading $-u^2$ and hence we expect a factor of $(1-u^2)$ in the result of the Fox derivative --- exactly the factor that has to be cancelled in order to get the multivariable Alexander polynomial. This shows that $\Delta(L')(u^2,v^2)$ can be directly computed by counting intersections of $\alpha$ with $[-2,-1]$ where relative gradings are computed as winding numbers of connecting paths around the special points $Y,X^-,X^+$, which count as $-v^2$, $-u^2$ and $-u^{-2}$, exactly as described by the specialization of the geometric algorithm for the colored HOMFLY complex. Thus $P = \Delta(L')$ and we are done. 
\end{proof}

\begin{rem}
The statement of the lemma can also be interpreted as saying that the geometric algorithm in section \ref{geom} computes the link Floer homology of $L$, because for rational links it contains exactly as much information as its multivariable Alexander polynomial. It would be interesting so see if the color stability of Conjecture 1 could be related to link Floer homology of a more general class of links with unknot components.
\end{rem}

\begin{proof}[Proof of Theorem \ref{MVAthm}]
We prove the theorem in two steps. First we compare the skein relations in the HOMFLY skein of the annulus and in an appropriate Alexander skein of the annulus. In the second step we use Lemma \ref{ratMVA} in the case of $(2,2k)$ torus links to compare the two polynomials on a common basis for the skeins. \\

For the first step pick a crossing $c$ in a diagram of $L'$ that does not involve strands in $U$ and denote by $L^+$, $L^-$ and $L^0$ the diagrams which have a positive crossing, a negative crossing and the oriented resolution of the crossing at position $c$ respectively. Because the involved strands are $\Lambda^1$-colored the colored HOMFLY polynomial $P$ satisfies:
\[a P(L^+)-a^{-1} P(L^-)= (q-q^{-1})P(L^0)\]
After substituting variables $a = 1, q = v, s = u/v$ we get exactly the skein relation for the multivariable Alexander polynomial for crossings whose strands are both labelled by $v^2$:
\[\Delta(L^+)-\Delta(L^-)= (v-v^{-1})\Delta(L^0)\]

From now on we assume that we have substituted variables $a = 1, q = v, s = u/v$ in all expressions. \\

Let $\Delta'(X)$ for a link $X$ in the annulus denote the evaluation of $X$ linked with an unknot via the skein theory of the multivariable Alexander polynomial, with $X$ labelled by $v^2$ and the unknot labelled by $u^2$. By virtue of the identical skein relation in the annulus, we can write $P^{st}(L',U) =  \sum_{I\in A} a_I t_i(\phi_I)$ and 
$\Delta(L) =  \sum_{I\in A} a_I \Delta'(\phi_I)$ with the same coefficients $a_I \in \mathbb{Z}[u^{\pm 1}, v^{\pm 1}]$ and with $\phi_I$ denoting a monomial in braid closures $\phi_i$ in the annulus.\\

\begin{lem} Let $\phi_I=\prod_{k=1}^{n_I} \phi_{i_{I,k}}$
then $\Delta'(\phi_I)= (1-u^2)^{n_I-1} \prod_{k=1}^{n_I} \Delta'(\phi_{i_{I,k}})$. 
\end{lem} 
\begin{proof}
At the expense of perhaps changing the label $v^2$ into $v^{-2}$ on some components $\phi_{i_{I,k}}$ we may assume that they are coherently oriented and hence $\phi_I$ can be written as a braid closure. The multivariable Alexander polynomial of a braid closure together with its axis can be computed via Theorem 1 in \cite{Mor}.
There $\Delta'(\phi_I)$ is presented as characteristic polynomial $\det(I-u^2 B(\phi_I))$ of a matrix $B(\phi_I)$ which is inductively built from a braid representative for $\phi_I$. It is easy to see that since $\phi_I$ is a disjoint union of $n_I$ braids, the matrix $B(\phi_I)$ has $n_I-1$ rows containing a single entry $1$ and zeros elsewhere. Removing all such rows (and the corresponding columns) via Laplace expansion, we get $\det(I-u^2 B(\phi_I)) = (1-u^2)^{n_I-1} \det(I-u B')$ where $B'$ is of block diagonal form and the blocks are exactly the matrices $B(\phi_{i_{I,k}})$.  
\end{proof}

Using the lemma we have:
\begin{align*}
&(1-u^2)\Delta(L')= (1-u^2)\sum_{I\in A} a_I \Delta'(\phi_I)= \sum_{I\in A} a_I (1-u^2)^{n_I} \prod_{k=1}^{n_I} \Delta'(\phi_{i_{I,k}}) \\
&= \sum_{I\in A} a_I  \prod_{k=1}^{n_I} (1-u^2)\Delta'(\phi_{i_{I,k}}) = \sum_{I\in A} a_I \prod_{k=1}^{n_I} t_i(\phi_{i_{I,k}})  = \sum_{I\in A} a_I t_i(\phi_I)= P^{st}(L',U)  
\end{align*}
Here the key step is that \[ (1-u^2)\Delta'(\phi_{i_{I,k}}) = P^{st}(\phi_{i_{I,k}}\cup U , U) = t_i(\phi_{i_{I,k}})\]
since $\phi_{i_{I,k}}$ linked with a meridional unknot is a $(2,2i_{I,k})$ torus link, for which theorem holds by Lemma \ref{ratMVA}.
\end{proof}

\subsection{Comparison with multivariable link invariants from Lie superalgebras}
Geer, Patureau-Mirand and Turaev, \cite{GP1}, \cite{GP2}, \cite{GPT} define multivariable polynomial link invariants using modified Reshetikhin-Turaev invariants for the Lie superalgebras $\mathfrak{sl}_{m|n}$. We give a brief review of this construction and how it is related to color stability of colored HOMFLY polynomials. 

Reshetikhin-Turaev $\mathfrak{sl}_{m|n}$ invariants are invariants of (framed) oriented tangles labelled by irreducible representations of the Lie superalgebra $\mathfrak{sl}_{m|n}$. They can be defined in a similar way as described in section \ref{RTi} by scanning the tangle in generic position from bottom to top and associating certain maps of $\mathfrak{sl}_{m|n}$ representations to cups, caps and crossings. While the representation theory of $\mathfrak{sl}_{m|n}$ is richer and more complicated than the representation theory of $\mathfrak{sl}_N$, it turns out that for most (to be precise: for so-called \emph{typical}) colorings, the resulting colored link invariants are trivial. The reason for this is that the quantum dimension of these $\mathfrak{sl}_{m|n}$ representations, and thus the invariants of unknots colored by such representations, are zero. 

The solution to this problem, as described in detail in \cite{GP2}, is to cut one component of the link $L$ open and consider it as an oriented two-ended tangle $T_\lambda$, where $\lambda$ indicates the representation on the open strand. The Reshetikhin-Turaev invariant of $T_\lambda$ is a multiple $x(T_\lambda) Id_{\lambda}$ of the identity map on the representation $\lambda$. Here $x(T_\lambda)$ lives in the ground ring $\mathbb{C}[[h]][h^{-1}]$ and $h$ is related to the familiar variable $q$ by $q=\exp(h/2)$. We can think of $x(T_\lambda)$ as an invariant of $L$ that is reduced with respect to the opened link component. Geer and Patureau-Mirand then show that there exist `fake quantum dimensions' $d(\lambda)\in \mathbb{C}[[h]][h^{-1}]$ of representations $\lambda$ that can be used to `fake unreduce' the invariant in a non-trivial\footnote{As noted earlier, unreducing with respect to actual quantum dimensions produces trivial invariants.} way:

\begin{thm} (Theorem 1 in \cite{GP2}) The map $L\mapsto F'(L):= d(\lambda) x(T_\lambda)$ is independent of the choice of cut component and typical representation $\lambda$ and, hence, is a well-defined framed colored link invariant.
\end{thm}

A great difference between the representation theories of $\mathfrak{sl}_{m|n}$ and $\mathfrak{sl}_N$ is that in the first case isomorphism classes of finite dimensional irreducible representations come in continuous families. To be more precise, for $\mathfrak{sl}_{m|n}$ they are indexed by $(d,z)\in \mathbb{N}^{m+n-2}\times \mathbb{C}$. Supposing that all colors on a link live in the same continuous family, it turns out that the invariants $F'(L)$ change in a very predictable way under varying these colors in their family:

\begin{thm} (Part of Theorem 2 in \cite{GP2})
\label{GPThm2} Let $L$ be a framed link with $k\geq 2$ components in some order, let $d\in \mathbb{N}^{m+n-2}$ and denote by $L(z_1,\dots, z_k)$ the link $L$ with components colored by the $\mathfrak{sl}_{m|n}$ representations indexed by $(d, z_i)$. 

Then there exists a framing independent invariant $M_{\mathfrak{sl}_{m|n}}^d(L) \in \mathbb{Z}[q^{\pm 1},q_1^{\pm 1}, \dots, q_k^{\pm 1}]$ of $L$ such that (up to renormalization) the following identity of Laurent series in $h$ holds:
\[F'(L(z_1, \dots , z_k)) = M_{\mathfrak{sl}_{m|n}}^d(L)(q, q_1 = q^{z_1}, \dots , q_k = q^{z_k}) \]
for all $z_i$ such that $(d,z_i)$ are typical representations. Note that we identify $q= \exp(h/2)$ and $q^z = \exp(z h /2)$.
\end{thm}

This theorem shows that the modified $\mathfrak{sl}_{m|n}$ Reshetikhin-Turaev invariants have very strong stability properties under shifting colors with respect to the continuous parameters $z_i$. In fact, Geer and Patureau-Mirand show that the color stability captured in $M_{\mathfrak{sl}_{m|n}}^d(L)$ is determined by the color stability of specializations of colored HOMFLY polynomials. In Proposition 3.4 of \cite{GP1} they prove that for certain integer values of $z_i$, the invariant $F'(L(z_1,\dots, z_k))$ agrees up to renormalization with the $a=q^{m-n}$ specialization of the  colored HOMFLY polynomial of $L$ labelled by Young diagrams that are determined by the pairs $(d, z_i)$; for details see the proof of Corollary 3.5 in \cite{GP1}. This together with Theorem \ref{GPThm2} shows:

\begin{prop} \label{GPprop}(Corollary 3.5 in \cite{GP1}) The multivariable link invariants $M_{\mathfrak{sl}_{m|n}}^d(L)$ are determined by and can in principle be computed from $a=q^{m-n}$ specializations of colored HOMFLY polynomials of $L$. 
\end{prop}

We expect that for a link $L$ with an unknot component the invariants $M_{\mathfrak{sl}_{m|1}}^0(L)$ are closely related to the color stable HOMFLY polynomial. In the case of $\mathfrak{sl}_{m|1}$ and $d=0$, the Young diagrams constructed in the proof of Proposition \ref{GPprop} represent exterior powers $\Lambda^i$:

\begin{cor} Let $L$ be a $k$-component link.  Then 
\[M_{\mathfrak{sl}_{m|1}}^0(L)(q, q_1 = q^{1+m-{i_1}},\dots, q_k = q^{1+m-{i_k}})\] agrees with the $a=q^{m-1}$ specialization of the $(\Lambda^{i_1},\dots, \Lambda^{i_k})$-colored HOMFLY polynomial of $L$ up to renormalization.
\end{cor}

The strong color stability properties described by $M_{\mathfrak{sl}_{m|n}}^d$ come at the price that (unlike the color stable HOMFLY polynomial) these invariants don't seem to be stable in super-rank $m-n$. In particular, we cannot expect colored HOMFLY polynomials of arbitrary links to be stable under changing color (in a simple way) without specializing $a=q^{m-n}$; this is already indicated by the decategorification of Conjecture \ref{conjB}. However, it is a very interesting question what information about the large color behaviour of HOMFLY type invariants of general links can be inferred from the color stability of related Lie superalgebra invariants.\\

Finally we want to mention another parallel between  $M_{\mathfrak{sl}_{m|1}}^0(L)$ and the color stable HOMFLY polynomial. Geer and Patureau-Mirand prove that their invariants specialize to the multivariable Alexander polynomial in a similar way as the color stable HOMFLY polynomial, see Theorem \ref{MVAthm}: 

\begin{thm}  (Theorem 3 in \cite{GP1})
The invariants $M_{\mathfrak{sl}_{m|1}}^0(L)$ of a $k$-component link $L$ specialize to the Conway potential function $\nabla(L)\in \mathbb{Q}(t_1, \dots, t_k)$ of $L$, which is a refinement of the multivariable Alexander polynomial:
\[\Delta(L)(q_1^{2 m},\dots, q_k^{2 m}) \sim \nabla(L)(q_1^m,\dots, q_k^m) = e^{\sqrt{-1}(m-1)\pi /2} M_{\mathfrak{sl}_{m|1}}^0(L)(q=  e^{\sqrt{-1}\pi /m}, q_1, \dots, q_k ) \]

Here $\sim$ means equality up to renormalization.  
\end{thm}

\end{document}